\documentclass[12pt]{extarticle}
\usepackage{amsmath, amsthm, amssymb, mathtools, hyperref, color}
\usepackage{graphicx}
\usepackage{caption}
\usepackage{subcaption}
\usepackage{tikz}
\usepackage{xparse}
\usepackage{tikz-cd}
\usepackage{tkz-graph}
\usetikzlibrary{shapes, arrows, positioning}
\usepackage{quiver}
\usepackage{algorithm}
\usepackage{algpseudocode}
\usepackage{lipsum}
\usepackage{tikz}
\usepackage{tikz-3dplot}
\usetikzlibrary{positioning}

\makeatletter
\let\@fnsymbol\@alph
\makeatother
\renewcommand{\thefootnote}{\alph{footnote}}

\title{Restricted Chip-Firing: Toric Toppling Ideals, Picard Groups, and Cellular Resolutions}
\author{Rahul Karki}

\begin{document}
\maketitle

% --- Unlabeled footnote for keywords & MSC ---
\begingroup
\renewcommand{\thefootnote}{} % Suppress label
\addtocounter{footnote}{1}    % Advance counter
\footnotetext{\textbf{Key words and phrases:} Chip-firing games on graphs, Picard group, toric ideals, Gr\"obner bases, free resolutions, hyperplane arrangements, Fano polytopes. 
\newline \textbf{MSC (2020):} 05C57, 14M25, 13P10, 13D02, 05C25, 13F55}
\endgroup
% --- Resume alphabetic footnotes ---
\renewcommand{\thefootnote}{\alph{footnote}}

% Theorem environments
\newtheorem{theorem}{Theorem}
\numberwithin{theorem}{section}
\newtheorem{proposition}[theorem]{Proposition}
\newtheorem{lemma}[theorem]{Lemma}
\newtheorem{corollary}[theorem]{Corollary}
\newtheorem{definition}[theorem]{Definition}
\newtheorem{remark}[theorem]{Remark}
\newtheorem{conjecture}[theorem]{Conjecture}
\newtheorem{problem}[theorem]{Problem}
\newtheorem{example}[theorem]{Example}
\newtheorem{question}{Question}

% Breakable algorithm environment
\makeatletter
\newenvironment{breakablealgorithm}
  {
    \begin{center}
    \refstepcounter{algorithm}
    \hrule height.8pt depth0pt \kern2pt
    \renewcommand{\caption}[2][\relax]{%
      {\raggedright\textbf{\fname@algorithm~\thealgorithm} ##2\par}%
      \ifx\relax##1\relax
        \addcontentsline{loa}{algorithm}{\protect\numberline{\thealgorithm}##2}%
      \else
        \addcontentsline{loa}{algorithm}{\protect\numberline{\thealgorithm}##1}%
      \fi
      \kern2pt\hrule\kern2pt
    }
  }{
    \kern2pt\hrule\relax
    \end{center}
  }
\makeatother

\begin{abstract}
We study certain groups and ideals arising from the chip-firing game on a generalisation of graphs called pargraphs. Several well-known families of toric ideals arise as toppling ideals of pargraphs. These include ideals defining rational normal curves, binomial edge ideals of complete graphs, and toric ideals of certain Fano polytopes. As our central result, we provide several equivalent conditions for the Picard group of a pargraph to be free, which equivalently characterise when its toppling ideal is toric. Furthermore, we construct a Gr\"obner basis for the toppling ideal, a minimal cellular free resolution for a distinguished initial ideal known as the $G$-parking function ideal, and establish Cohen-Macaulay property for these ideals.
\end{abstract}

\section{Introduction}\label{Picard group}
Let $\mathbb{K}$ be a field and $R_n=\mathbb{K}[x_1,\dots,x_n]$ be the polynomial ring in $n$-variables over $\mathbb{K}$. An ideal $I$ of $R_n$ is called \emph{toric} if it is a prime ideal generated by binomials of the form ${\bf x^{u}}-{\bf x^{v}}$ where ${\bf u, v}\in \mathbb{N}^{n}$ (\cite[Subsection 3.1]{herzog2018binomial}). The theory of toric ideals plays a central role in several areas of mathematics, due to their rich structural properties and significant applications in geometry, optimization, and algebraic statistics \cite{miller2005combinatorial},\cite{herzog2018binomial},\cite{sturmfels1996grobner}.
 In recent years, toric ideals arising from graphs have been extensively studied, investigating their robustness properties and connections to numerical semigroups, Koszul algebras, and algebraic statistics via Markov chains \cite{herzog2010binomial},\cite{greif2020green},\cite{garcia2023robustness}, \cite{almousa2025root}. In this work, we study the toric structure of ideals arising from the \emph{chip-firing game} on a generalisation of graphs, which we call \emph{pargraphs}.

Let $G=(V(G),E(G))$ be a connected multigraph, possibly with multiple edges but no loops. The chip-firing game is a discrete dynamical system on $G$. The game starts with an initial assignment of integer entries to the vertices of $G$ that are known as chips or dollars. At each move of this game, some vertex $v$ fires one chip along all the edges incident on it. This results in $v$ losing its degree or valence number of chips, with each other vertex $u$ receiving a number of chips equal to the number of edges between $u$ and $v$. We refer to \cite{corry2018divisors},\cite{klivans2018mathematics}, Subsection \ref{chipfireParg} for more details on this game.
%Given an initial configuration, a fundamental question is whether there exists a finite sequence of chip-firing moves that leads to a configuration with non-negative chips at every vertex. In general, the answer is no and depends subtly on the initial configuration. 
The chip-firing game and its variants have a very rich theory and have been studied from multiple perspectives, including geometric (connected to Riemann-Roch and Brill-Noether theories \cite{BAKER2007766},\cite{baker2008specialization}), combinatorial (such as stabilization and recurrence analysis \cite{klivans2018mathematics}), algebraic (Picard and Jacobian groups, toppling and $G$-parking function ideals \cite{BAKER2007766},\cite{ChipFirePotential}, \cite{manjunathschwil}), and statistical physics where it is well known as the Abelian Sandpile Model. Here, our primary interest lies on the algebraic aspects of this game on pargraphs. 

Informally, a pargraph is a combinatorial structure arising from a partition of the vertex set of a graph $G$. More precisely,
a pargraph is an ordered pair $(G,\Pi)$ where $G$ is a connected multigraph and $\Pi=\{V_1,\dots,V_k \}$ is a partition of the vertex set $V(G)$ such that the subgraph induced on each $V_i$ is connected. %When the connectedness requirement on the induced subgraphs $G[V_i]$ is removed, then the ordered pair $(G,\Pi)$ is called a \emph{pseudo-pargraph}.
%The vertices of $G$ are called the \emph{basic vertices} of $(G,\Pi)$, and 
The elements of $\Pi$ are called the \emph{main vertices} of the pargraph. An edge $\{i,j\}$ of $G$ is called a \emph{relevant edge} of the pargraph if $i$ and $j$ lie in two different main vertices of $(G,\Pi)$. %If $i$ and $j$ lie in the same main vertex, then $\{i,j\}$ is referred to as a \emph{basic edge} of the pargraph. %The edges of $G$ between two vertices of $G$ that lie in the same main vertex are called \emph{basic edges} of $(G,\Pi)$, and the edges between two basic vertices that lie in distinct main vertices are called \emph{relevant edges} of the pargraph. 
 We refer to Figure \ref{egpargaph} for an example of a pargraph. %The edges $\{1,2\}, \{5,7\}$ are examples of  relevant edges of $(G,\Pi)$, respectively.

\begin{figure}[ht]
\begin{center}
\begin{tikzpicture}[scale=0.8, every node/.style={inner sep=1.5pt}]

\begin{scope}[xshift=7cm]
  \coordinate (c1) at (2.5,2);
  \coordinate (c2) at (5,2);
  \coordinate (c3) at (0,0);
  \coordinate (c4) at (2.5,0);
  \coordinate (c5) at (5,0);
  \coordinate (c6) at (2.5,-2);
  \coordinate (c7) at (5,-2);

  \foreach \a/\b in {1/2,2/5,3/1,3/4,3/6,4/5,6/7,7/5}{\draw[thick] (c\a)--(c\b);}
  \foreach \i/\pos in {1/above,2/above,3/above,4/above,5/right,6/above,7/below}{
    \node[circle,draw,fill,minimum size=6pt,label=\pos:{\small \i}] at (c\i) {};
  }

  % ellipse around 2 and 5
  \draw[blue, thick, dashed] (5,1) ellipse (0.6cm and 1.5cm);
\end{scope}
\end{tikzpicture}
\end{center}
\caption{\label{egpargaph}The pargraph $(G,\Pi)$ with $\Pi=\{ \{2,5\},\{j\}\mid j \in \{1,3,4,6,7\}
\}$.}
\end{figure}

The algebraic approach proceeds by associating certain groups and ideals to the game, which we briefly outline here. The chip-firing game on a pargraph $(G,\Pi)$ is similar to playing the game on $G$ where we are only allowed to fire from the main vertices of $(G,\Pi)$. In the chip-firing from a  main vertex $V_i\in \Pi$, all vertices of $G$ in $V_i$ fire chips simultaneously (see Subsection \ref{chipfireParg}). 
These firing moves are encoded in a lattice (that is, a subgroup of $\mathbb{Z}^{n}$) $L_{\Pi}$ that we call as the \emph{Laplacian lattice} of the pargraph (Definition \ref{deflaplattice}).
%The algebraic approach proceeds by associating certain groups and ideals to the game, which we briefly outline here. The chip-firing game on a pargraph $(G,\Pi)$ is similar to playing the game on $G$ where we are only allowed to fire from the main vertices of $(G,\Pi)$. %elements of $\Pi$. 
%In the chip-firing from a  main vertex $V_i\in \Pi$, all vertices of $G$ in $V_i$ fire chips simultaneously. These firing moves are encoded in a lattice (that is, a subgroup of $\mathbb{Z}^{n}$) $L_{\Pi}$ that we call as the \emph{Laplacian lattice} of the pargraph. Let $V(G)=\{1,\dots,n\}$ be the set of vertices of $G$. Consider the Laplacian matrix $\Lambda_G=D-A_G$ of the graph $G$ where $D$ is the diagonal matrix with $i$-th diagonal entry being the valence or degree of the vertex $i$, and $A_G$ is the adjacency matrix whose $i,j$-th entry represents the total number of edges between vertices $i$ and $j$. In case of the chip-firing game on $G$, the columns of $\Lambda_G$ encodes the chip-firing moves. Let $b_1,\dots,b_n$ be the columns of $\Lambda_G$ corresponding to the vertices $1,\dots,n$, respectively. The Laplacian lattice $L_{\Pi}$ of the pargraph is the lattice generated by the vectors $\{b_{V_1},\dots,b_{V_k}\}$ where $b_{V_i}=\sum_{j\in V_i}b_j$. 
%Let $\mathbb{K}$ be a field and $R_n=\mathbb{K}[x_1,\dots,x_n]$ be the polynomial ring in $n$-variable over $\mathbb{K}$. 
We refer the lattice ideal $I_{\Pi}$ associated with $L_{\Pi}$ as the \emph{toppling ideal} of the pargraph. In particular, $I_{\Pi}=\langle {\bf x}^{\bf u}-{\bf x}^{\bf v} \mid {\bf u, v} \in \mathbb{N}^{n} \text{ and } {\bf u}-{\bf v}\in L_{\Pi} \rangle$. Many of the algebraic and geometric properties of the ideal $I_{\Pi}$ are encoded in the quotient group $\mathbb{Z}^{n}/L_{\Pi}$, which we refer to as the \emph{Picard group} of the pargraph. These definitions extend the classical graph-theoretic notions of the Picard group and toppling ideal to pargraphs.

The study of toppling ideal for graphs started with the work of Cori, Rossin, and  Salvy \cite{CoriRossinSalvy}, where they introduced an inhomogeneous version of the toppling ideal and constructed a Gr\"obner basis of it, establishing connections to \emph{stable configurations} and the \emph{sandpile group}.
Subsequently, the homogeneous version of the toppling ideal was investigated by Perkinson, Perlman, and Wilmes \cite{perkinson2013primer},  Manjunath and Sturmfels \cite{manjunath2013monomials}, Manjunath, Schreyer, and Wilmes \cite{manjunathschwil}. 
These works explored algebraic and geometric properties including the construction of Gr\"obner bases, explicit minimal free resolutions, and the study of initial ideals, notably the $G$-parking function ideal. 

For graphs, it is well-known that the toppling ideal of a graph is toric if and only if $G$ is a tree {\cite[Propositions 1.20, 2.37]{corry2018divisors}}. Furthermore, the toppling ideal of a tree is generated by homogeneous binomials of degree $1$. But in \cite{karki2024rationalnormalcurveschip}, the author and Manjunath proved that the extension of toppling ideals to \emph{parcycles} %(a special case of pargraph where $G$ is a cycle graph) 
yields a family of toric ideals that includes rational normal curves and binomial edge ideals of complete graphs. This work naturally leads to the question of whether higher-dimensional Veronese embeddings and more general determinantal ideals can be realised as toppling ideals corresponding to such structures.
In this work, we take preliminary steps toward this objective by investigating the following questions, which are the main focus of this article: 1. Determining conditions under which the toppling ideal of a pargraph is toric. 2. Constructing a Gr\"obner basis for the toppling ideal. 3. Computing algebraic invariants and investigating the Cohen-Macaulay property. The answers to the above questions lead to examples in Section \ref{FanChip} of toric ideals of polytopes that appear as toppling ideals.

We also study the Picard group of a pargraph, with a primary interest in establishing conditions for its freeness.
Understanding this property is instrumental in characterising when the corresponding toppling ideal is toric. The study of the Picard group (also known as the sandpile group, critical group) of a graph has been an active area of research. In recent years, several studies have focused on computing the Picard groups of directed graphs, including trees, cycles, wheels, and neural network graphs, examining free and torsion components \cite{jun2025picard},\cite{alfaro2021structure},\cite{Fittingidealtakenori}. One of our main results in this direction is the following. Consider the reduced Laplacian polynomial of a pargraph $(G,\Pi)$ (which we define formally in Definition \ref{laplacianpoly}).
\begin{theorem}\label{mainthm}
   The Picard group of a pargraph $(G,\Pi)$ is free if and only if  the content of its reduced Laplacian polynomial is $1$.
\end{theorem}
An illustration of the applicability of Theorem \ref{freeness of picard group} is provided in Examples \ref{exampofpicard}, \ref{mainthmexam}. As a consequence,
Corollary \ref{corrpargrah} establishes another necessary and sufficient condition for the freeness of the Picard group in terms of pargraphs, and Corollary \ref{picGfreethm} provides an additional, purely structural sufficient condition.
%We illustrate the usefulness of Theorem \ref{freeness of picard group} in Example \ref{exampofpicard}. Consequently, we establish another (sufficient) intrinsic criterion for the freeness of the Picard group in Corollary \ref{picGfreethm}.
%In Example \ref{exampofpicard}, we show that the hypothesis in Theorem \ref{picGfreethm} is a sufficient, but not necessary, condition for the Picard group of a pargraph to be free. Consequently, we establish another criterion for the freeness of the Picard group in Lemma \ref{freeness of picard group}.}

We also provide an explicit description of a Gröbner basis for the toppling ideal by extending the divisor theory on graphs to pargraphs. Furthermore, we construct a cellular minimal free resolution of the corresponding initial ideal, called the \emph{$G$-parking function ideal} of the pargraph (Subsections \ref{secgboftop}, \ref{cellresofparkingidelas}). The results are as follows.

\begin{theorem}\label{gboftop}
    The set of binomials $\{{\bf x}^{S \rightarrow \overline{S}}-  {\bf x}^{\overline{S} \rightarrow S} \mid S  \text{ and }\overline{S} \text{ are connected}\newline \text{parsets of } V(G)\text{ and } V_k\subseteq \overline{S}\}$ forms a Gr\"obner basis of the toppling ideal $I_{\Pi}$ with respect to the monomial order $m_{rev}$.
\end{theorem}
\begin{theorem}\label{cellres}
     Let $(G,\Pi)$ be a pargraph and $M_{\Pi}$ be the corresponding $G$-parking function ideal. The cellular free resolution $\mathcal{H}_{\Pi}$ supported on the labelled polyhedral cell complex $\mathcal{B}_{\Pi}$ is a minimal free resolution of the quotient ring $R_{n-1}/M_{\Pi}$.
 \end{theorem}

%The construction of the above Gr\"obner basis for the toppling ideal yields an initial ideal that extends the notion of the $G$-parking function ideal from graphs to pargraphs. In Subsection \ref{cellresofparkingidelas}, we construct a minimal cellular free resolution of the $G$-parking function ideal of a pargraph, and establish its Cohen–Macaulay property.

{\bf Acknowledgements:} The author thanks Kiran Kumar A.S. for several helpful discussions related to the Picard group. We used Macaulay2 for some of the computations in Sections \ref{AsP} and \ref{FanChip}.

\section{ Combinatorics of Pargraphs}\label{preparcyc_subsect}
Let $G=(V(G),E(G))$ be a connected multigraph with vertex set $V(G)=\{1,\dots,n\}$.
%Let $G=(V(G),E(G))$ be a connected multigraph (with no loops) on the set of vertices $V(G)=\{1,\dots,n\}$ and let $E(G)$  be the edge set of $G$. 
For a subset $S$ of $V(G)$, we use the notation $G[S]$ to denote the subgraph of $G$ induced by $S$.
A partition $\Pi=\{V_1,\dots,V_k\}$ of $V(G)$ is called a \emph{connected $k$-partition} (or simply a \emph{connected partition}) if $G[V_i]$ is connected for all $i$.  
\begin{definition}\begin{enumerate}
    \item A pargraph is an ordered pair $(G,\Pi)$ where $G$ is a graph and $\Pi=\{V_1,\dots,V_k\}$ is a connected partition of $V(G)$.
 
\item We call $(G,\Pi)$ a pseudo-pargraph if $G$ or some of the induced subgraphs $G[V_i]$ are not necessarily connected.
\end{enumerate}
\end{definition}
    
 We call the vertices of $G$ as the \emph{basic vertices} of $(G,\Pi)$ and the elements of $\Pi$ as the \emph{main vertices} of $(G,\Pi)$. An edge of $G$ that connects two basic vertices of $(G,\Pi)$ that lie in the same main vertex is called a \emph{basic edge} of $(G,\Pi)$.
%An edge $\{i,j\}$ of $G$ is called a \emph{relevant edge} of $(G,\Pi)$ if $i\in  V_i, j\in V_j$, and $V_i \neq V_j$. 
%{\color{blue}Recall the notion  of relevant edges of a pargraph defined in the introduction}. %Basic edges of $(G,\Pi)$ give information about the connectivity of basic vertices in a main vertex of $(G,\Pi)$, and 
The remaining edges of $(G,\Pi)$ are called relevant edges (defined in the introduction), which connect basic vertices lying in two
different main vertices of $(G,\Pi)$. For example, the set of basic edges of $(C_7,\Pi)$ (as in Example \ref{exa23}) is $\{\{1,7\},\{6,7\},\{3,4\}\}$, and the set of relevant edges is \{\{1,2\},\{2,3\},\{4,5\},\{5,6\}\}.

We denote the Laplacian matrix of $G$ by $\Lambda_{G}$. It is defined as $\Lambda_{G}=D-A_G$ where $D$ is the diagonal matrix whose $(i,i)$-entry is the valence or degree of the vertex $i$. The matrix $A_G=(a_{i,j})$ is called the \emph{adjacency matrix} of $G$ where $a_{i,j}$ denotes the total number of edges between the vertices $i$ and $j$.

%Let $b_j$ be the row of the Laplacian matrix $\Lambda_{G}$ of $G$ that corresponds to the vertex $j\in V(G)$. We define the \emph{Laplacian lattice} of $(G,\Pi)$ as the lattice generated by the vectors $b_{V_1},\dots,b_{V_k}$ where $b_{V_i}:=\sum_{j\in V_i}b_j\in \mathbb{Z}^{n}$.

\begin{definition}\label{deflaplattice}
\begin{enumerate}
    \item Let $b_j$ be the row of the Laplacian matrix $\Lambda_{G}$ of $G$ that corresponds to the vertex $j\in V(G)$. We define the \emph{Laplacian lattice} of $(G,\Pi)$ as the lattice generated by the vectors $b_{V_1},\dots,b_{V_k}$ where $b_{V_i}:=\sum_{j\in V_i}b_j\in \mathbb{Z}^{n}$.
    \item A subset $S$ of $V(G)$ is called a parset of $(G,\Pi)$ if $S=\cup_{i \in U} V_i$ where $U \subseteq \{1,\dots,k\}$. A parset $S$ of $(G,\Pi)$ is called a connected parset if $G[S]$ is connected.
\end{enumerate}

\end{definition}

The notions of parsets and connected parsets are useful in the study of the chip-firing game on pargraphs. We will later use these notions when we deal with certain ideals called toppling ideals and $G$-parking function ideals associated with pargraphs.

For a basic vertex $i\in V_i$, let $E(i)=\{\{i,j\}\mid j\in V_i \text{ and } i\neq j\}$ denote the set of all possible unordered pairs containing $i$ and an element $j (\neq i)\in V_i$. Note that the set of basic edges of $(G,\Pi)$ that are incident on $i$ is a subset of $E(i)$. Let ${{\rm rel}_{E}(i)}$ be the set of all relevant edges of $(G,\Pi)$ that are incident on $i$, and ${{\rm rel}_{E}(i,j)}$ be the set of relevant edges between vertices $i$ and $j$.

For any two basic vertices $i,j$ in a main vertex $V_i$, let
$x_e$ be the variable corresponding to $e=\{i,j\}$. Consider the weighted Laplacian matrix 
$L(\mathbf{x})$ which is defined as \begin{equation}
(L(\mathbf{x}))_{ij} = 
\begin{cases} 
-x_{e}, & \text{if } i \neq j \text{ and } i,j \text{ lie in the same main vertex} \\ 
-|{\rm rel}_{E}(i,j)|, & \text{if } i \neq j \text{ and } i,j \text{ lie in distinct main vertices} \\ 
|\mathrm{rel}_{E}(i)| + \sum_{e \in E(i)} x_e, & \text{if } i = j 
\end{cases}
\end{equation}
Since the pargraphs we consider arise from undirected graphs, we assume  $x_{ij}=x_{ji}$ for all $e=\{i,j\}$. Note that the weighted Laplacian matrix $L(\mathbf{x})$ can be written as sum of two matrices $L_{R}$ and $L_{B}$ where 
\begin{equation}
(L_R)_{ij} = 
\begin{cases} 
0, & \text{if } i \neq j \text{ and } i,j \text{ lie in the same main vertex} \\ 
-|{\rm rel}_{E}(i,j)|, & \text{if } i \neq j \text{ and } i,j \text{ lie in distinct main vertices} \\ 
|\mathrm{rel}_{E}(i)|, & \text{if } i = j 
\end{cases}
\end{equation}
and 
\begin{equation}
(L_B)_{ij} = 
\begin{cases} 
-x_{e}, & \text{if } i \neq j \text{ and } i,j \text{ lie in the same main vertex} \\ 
0, & \text{if } i \neq j \text{ and } i,j \text{ lie in distinct main vertices} \\ 
\sum_{e \in E(i)} x_e, & \text{if } i = j 
\end{cases}.
\end{equation}
The matrix $L_R$ encodes information about the relevant edges of $(G,\Pi)$.
\begin{definition}\label{laplacianpoly} Let $L(\bf{x})$ be the weighted Laplacian matrix of a pargraph $(G,\Pi)$ and let $(L(\bf{x}))_{\rm red}$ be the matrix obtained from $L(\bf{x})$ by removing the last row and column. We define:
\begin{enumerate}
    \item The reduced Laplacian matrix of $(G,\Pi)$ to be the matrix  $(L(\bf{x}))_{\rm red}$.
    \item The reduced Laplacian polynomial of $(G,\Pi)$ to be the polynomial\newline ${\rm det}((L{(\bf{x}}))_{\rm red}) \in \mathbb{Z}[{\bf x}]=\mathbb{Z}[x_e \mid e\in E(i), i\in V(G)] $.
\end{enumerate}
\end{definition}

\begin{example}\label{exa23}{\rm Consider the pargraph $(C_7,\Pi)$ where $C_7$ is the cycle graph on $7$-vertices and $\Pi=\{\{1,6,7\},\{2\},\{3,4\},\{5\}\}$. For the basic vertex $1$, $E(1)=\{ \{1,7\},\{1,6\}\}, \mathrm{rel}_{E}(1)=\{\{1,2\}\}, {\rm rel}_{E}(1,2)=\{\{1,2\}\},$ and ${\rm rel}_{E}(1,3)$ is the empty set.
\begin{center}
\begin{tikzpicture}[
    scale=0.65,
    every node/.style={font=\sffamily\small},
    vertex/.style={circle, fill=black, inner sep=1.5pt}
]
    % Radius of the heptagon
    \def\R{1.5}

    % Vertices (C_7 cycle arranged symmetrically)
    \node[vertex, label=below:7]        (v7) at (90:\R) {};
    \node[vertex, label=above left:1]   (v1) at (90 + 1*360/7:\R) {};
    \node[vertex, label=left:2]         (v2) at (90 + 2*360/7:\R) {};
    \node[vertex, label=below:3]        (v3) at (90 + 3*360/7:\R) {};
    \node[vertex, label=below:4]        (v4) at (90 + 4*360/7:\R) {};
    \node[vertex, label=right:5]        (v5) at (90 + 5*360/7:\R) {};
    \node[vertex, label=above right:6]  (v6) at (90 + 6*360/7:\R) {};

    % Cycle edges
    \draw[thick] (v1) -- (v2) -- (v3) -- (v4) -- (v5) -- (v6) -- (v7) -- (v1) -- cycle;

    % Top boundary enclosing {1, 7, 6}
    \draw[thick] (0, 1.1) ellipse (1.60cm and 0.60cm);

    % Bottom boundary enclosing {3, 4}
    \draw[thick] (0, -1.3) ellipse (1.05cm and 0.45cm);

\end{tikzpicture}
\end{center}
In this case, the weighted Laplacian matrix $$L(\bf{x})=\begin{bmatrix}
1 + x_{16} + x_{17} & -1 & 0 & 0 & 0 & -x_{16} & -x_{17} \\
-1 & 2 & -1 & 0 & 0 & 0 & 0 \\
0 & -1 & 1 + x_{34} & -x_{34} & 0 & 0 & 0 \\
0 & 0 & -x_{34} & 1 + x_{34} & -1 & 0 & 0 \\
0 & 0 & 0 & -1 & 2 & -1 & 0 \\
-x_{16} & 0 & 0 & 0 & -1 & 1 + x_{16} + x_{67} & -x_{67} \\
-x_{17} & 0 & 0 & 0 & 0 & -x_{67} & x_{17} + x_{67}
\end{bmatrix}$$ and the reduced Laplacian polynomial of $(C_7,\Pi)$ is $4x_{16}x_{17}x_{34} + 4x_{16}x_{34}x_{67} + 4x_{17}x_{34}x_{67} + x_{16}x_{17} + x_{17}x_{34} + x_{16}x_{67} + x_{17}x_{67} + x_{34}x_{67} \in \mathbb{Z}[x_{16},x_{17},x_{34},x_{67}]$}.\qed
\end{example}

Any binary assignment of the variables $x_{ij}\in \{0,1\}$ naturally defines a graph $H$ on $V(G)$ such that the Laplacian matrix $\Lambda_{H}$ is the specialisation of $L(\bf{x})$ under this assignment.  Note that $H$ might not be connected. Moreover, the pseudo-pargraph $(H,\Pi)$ has the same Laplacian lattice as $(G,\Pi)$. This highlights the importance of the notion of a weighted Laplacian matrix, which helps us construct a family of pseudo-pargraphs having the same Laplacian lattice as $(G,\Pi)$.

\subsection{Chip-Firing Game on Pargraphs}\label{chipfireParg}
%We start with the chip-firing game on a connected multigraph $G$. The chip-firing game is a dynamical system on $G$. The game starts with an assignment of integer entries, called an initial configuration, to each vertex of $G$. These entries are called chips or dollars. At each step of the game, some vertex, say $i$, sends one chip along all the edges incident to it. As a consequence of this, the vertex $i$ loses its valence numbers of chips, and every other vertex $j$ receives $a_{i,j}$ chips where $a_{i,j}$ is the total number of edges between vertices $i$ and $j$. 
Recall the chip-firing game on a multigraph defined in the introduction. The initial assignment of integer entries to the vertices of $G$ is called an \emph{initial configuration}.
A natural question that arises in the chip-firing game is: for any given initial configuration, does there exist a finite sequence of chip-firing moves such that, in the end, all the vertices have a non-negative number of chips? In general, the answer is no. Note that the total number of chips at each stage of the chip-firing game remains the same. Therefore, for an initial configuration with a negative total number of chips, it is not possible to find such a finite sequence of chip-firing moves. For a configuration with total number of non-negative chips, the existence of such a sequence subtly depends on the structure of the underlying graph. We refer to \cite{corry2018divisors},\cite{klivans2018mathematics}
for more details on the chip-firing game on a graph.

The chip-firing game on a pargraph $(G,\Pi)$ is similar to the chip-firing game on $G$, but here we are only allowed to fire from the parsets of the pargraph $(G,\Pi)$. In the chip-firing from a parset $S$, all vertices of $G$ in $S$ fire chips simultaneously, and this can also be realised as a sequential chip-firing process from the main vertices whose union forms $S$.
A \emph{divisor} (or \emph{configuration}) on $(G,\Pi)$ is an assignment of an integer to every basic vertex of $(G,\Pi)$. Every divisor $D$ on $(G,\Pi)$ defines a vector $(d_1,\dots,d_n)\in \mathbb{Z}^{n}$ where $d_i$ denotes the integer assigned to the vertex $i\in V(G)$ by the divisor $D$. We define the \emph{degree} of a divisor $D=(d_1,\dots,d_n)$ as ${\rm deg }(D):=\sum_{i=1}^{n}d_i$.
Two divisors $D_1$ and $D_2$ on $(G,\Pi)$ are said to be \emph{chip-firing equivalent} if one is obtained from another by a finite sequence of chip-firing moves on the pargraph. In particular, $D_1$ and $D_2$ on $(G,\Pi)$ are chip-firing equivalent if and only if $D_1 - D_2 \in L_{\Pi}$.
We call a divisor $D$ on $(G,\Pi)$ an \emph{effective divisor} if $D(i)\geq 0$ for all $1\leq i\leq n$. For a divisor $D$ on $(G,\Pi)$, a parset $S$ is called \emph{legal} to fire if $D(i)\geq {\rm outdeg}_{S}(i)$ for all $i\in S$, where ${\rm outdeg}_{S}(i)$ is the number of edges between the vertex $i$ and the vertices in $V(G)\setminus S$.  

\begin{definition}
A divisor $D$ on $(G,\Pi)$ is called a $V_{k}$-reduced divisor if it satisfies the following properties.
\begin{enumerate}
    \item $D(i) \geq 0$ for all $i\in V(G)\setminus V_k$;
    \item for every parset $S$ of $V(G)$ that does not contain $V_k$, there exists a (basic) vertex $j\in S$ such that $D(j)<{\rm outdeg}_{S}(j)$. %where ${\rm outdeg}_{S}(j)$ is the number of edges between $j$ and $V(G)\setminus S$.
\end{enumerate}
         \end{definition}

Let $T$ be a spanning tree of $G$ such that $T[V_i]$ is a spanning tree of $G[V_i]$ for all $1\leq i \leq k$.
The existence of such a spanning tree of $G$ holds, as $G$ is connected and $\Pi$ is a connected partition of $V(G)$. For a basic vertex $v\in V_k$, the pair $(T,v)$ is called a \emph{spanning tree of the pargraph} $(G,\Pi)$ rooted at the basic vertex $v$.

Let $v_1,\dots,v_n$ be an ordering of the basic vertices of $(G,\Pi)$ in such a way that the index  $i<j$ if $v_i$ lies on the unique path on $(T,v)$ connecting $v_j$ to $v$. In this ordering, $v_1=v$ and such an ordering of the basic vertices is called a \emph{tree ordering} compatible with the spanning tree $T$. A tree ordering of the basic vertices induces a total order $<'$ on the set of divisors on $(G,\Pi)$. For divisors $D_1$ and $D_2$ on $(G,\Pi)$,   $D_1<'D_2$ if either  ${\rm deg}(D_1)<{\rm deg}(D_2)$, or ${\rm deg}(D_1)={\rm deg}(D_2)$ and $D_1(v_i)>D_2(v_i)$ where $i$ is the smallest index such that $D_1(v_i)\neq D_2(v_i)$. Informally, in this ordering, every legal firing from a parset not containing $V_k$ sends some chips ``closer'' to $V_k$. As a consequence of this, the new divisor that we obtain will always be smaller than the previous one.

\begin{lemma}\label{effdiv} Every effective divisor on $(G,\Pi)$ is chip-firing equivalent to a unique $V_k$-reduced divisor.
    \end{lemma}
\begin{proof}
    Let $D$ be an effective divisor on $(G,\Pi)$. If $D$ is not $V_k$-reduced, then we perform a sequence of legal chip-firing moves from the parsets of $(G,\Pi)$ not containing $V_k$. Let $D_1(=D),D_2,\dots$ be the sequence of effective divisors obtained as a consequence of the above legal chip-firing moves. Each effective divisor $D_i$ can be expressed as $\tilde{D}_i+\sum_{v_l\in V_k}D_i(v_l)[v_l] $ where  $\tilde{D}_i(v_j)=D_i(v_j)\chi_{(V(G)\setminus V_k)}(v_j)$. Note that ${\rm deg}(\tilde{D}_1 )\geq {\rm deg}(\tilde{D}_i)\geq {\rm deg}(\tilde{D}_{i+1})\geq 0 $, and $D_{i+1}<'D_{i}$ for all $i$. Hence, there will only be finitely many distinct divisors in the sequence, and it stops after a finite number of legal chip-firing moves. As a consequence of this, we obtain a $V_k$-reduced divisor that is chip-firing equivalent to $D$.
    
    Suppose that there exist two $V_k$-reduced divisors, say $D_1$ and $D_2$, that are chip-firing equivalent to $D$. This implies the divisor $D_1-D_2\in L_{\Pi}$, i.e. $D_2=D_1- \Lambda_{G}\sigma$ where $\sigma\in \mathbb{Z}^{n}$ such that $\sigma(v_i)=\sigma(v_j)$ if $v_i,v_j\in V_l$ for some $l$. If $\sigma$ is a constant function, then $D_1=D_2$. Suppose that $\sigma$ is not a constant function. Let $m={\rm max}\{\sigma(v_i)\mid v_i\in V(G)\}$.
    Let $S=\{v_i\in V(G)\mid \sigma(v_i)=m\} $. Note that $S$ is a parset of $V(G)$. Furthermore, we can assume that $S\subseteq V(G)\setminus V_k$ (otherwise, we work with $D_1=D_2-\Lambda_{G}\sigma$). If $v_i\in S$, then $\sigma(v_i)-\sigma(v_j)\geq 1$ for $v_j\notin S$. Hence,
    $0\leq D_2(v_i)=D_1(v_i)-\sum_{(v_i,v_j)\in E(G)}(\sigma(v_i)-\sigma(v_j))\leq D_1(v_i)-{\rm outdeg}_{S}(v_i)$. This implies that we can legally fire from the parset $S$ which gives us a contradiction as $D_1$ is a $V_k$-reduced divisor. Hence, the chip-firing equivalence class of each effective divisor contains a unique $V_k$-reduced divisor.
\end{proof}

\section{Algebraic Aspects of Pargraphs}\label{AsP}
Throughout this section, we assume that $G$ is a connected multigraph on the vertex set $V(G)=\{1,\dots,n\}$, and $\Pi=\{V_1,\dots,V_k\}$ is a connected partition of $V(G)$.

\subsection{Picard Group of a Pargraph}

Let $D$ be a divisor on $(G,\Pi)$ and let $[D]$ be the chip-firing equivalence class of $D$. For any two divisors $D_1$ and $D_2$, we define $D_1+D_2=\sum_{i\in V(G)}(D_1(i)+D_2(i))[i]$. Note that if $D_1\in [D_2]$ and $D_3\in [D_4]$, then $D_1+D_3\in [D_2+D_4]$. Under the above additive operation, the set of all chip-firing equivalence classes of divisors on $(G,\Pi)$ forms an abelian group and is isomorphic to ${\mathbb{Z}}^{n}/L_{\Pi}$. We call this group as the Picard group of the pargraph $(G,\Pi)$ (as defined in the introduction). In this subsection, we study the freeness of this group in order to understand the toric nature of the toppling ideal $I_{\Pi}$.

From the structure theorem for finitely generated abelian groups, we know that the Picard group of $(G,\Pi)$ is a direct sum of a free abelian group and a torsion group, which we refer to as the \emph{Jacobian subgroup} of $(G,\Pi)$. A standard technique used to study the structure of a quotient group $\mathbb{Z}^n/A$ is to compute the Smith normal form of a matrix $M$ whose columns generate the subgroup $A$ \cite[Chapter 7]{miller2005combinatorial}, \cite{jun2025picard}. In the following, we briefly recall the notion of the Smith normal form of a matrix. %and some results associated with it.}
%In case of graphs, we know that the Picard group of a graph ${\rm Pic}(G)$ is free if and only if $G$ is a tree. In the following, we provide a sufficient condition for the Picard group of a pargraph to be free. Let $G_l$ be a connected graph with $V(G_l)=V(G)$. For each $V_i\in \Pi$, we define $b_{l,V_i}=\sum_{v_j \in V_i}b_{l,v_j}$ where $b_{l,v_j}$ is the column (or row) corresponding to the vertex $v_j$ in the Laplacian matrix of $G_l$.
%\begin{theorem}\label{freeness of picard group}
 %   Let $(G,\Pi)$ be a pargraph where $\Pi=\{V_1,\dots,V_k\}$ is a connected partition of $V(G)$. Let $G_1,\dots,G_m$ be a finite collection of connected graphs on the same vertex set as $G$, and let $d_{1j},\dots,d_{nj}$ be the invariant factors (in Smith normal form) of $G_j$. If $b_{l,V_t}=b_{V_t}$ for all $1\leq t \leq k, 1\leq l \leq m$, and ${\rm gcd}\{d_{(n-1)j}\mid  1\leq j \leq m\}=1$, then the Picard group of $(G,\Pi)$ will be free.
%\end{theorem}
%\section{Preliminaries}

Let $R$ be a principal ideal domain and let $M_{m,n}(R)$ be the set of all $m\times n$ matrices with entries in $R$. Two matrices $A,B\in M_{m,n}(R)$ are called \emph{equivalent} if there exist invertible matrices $U\in M_{m,m}(R)$ and $W\in M_{n,n}(R)$ such that $A=UBW$.

\begin{definition}\label{defsminf}
    The Smith normal form of a matrix $A\in M_{m\times n}(R)$ is a matrix $Q=(q_{i,j})\in M_{m\times n}(R)$ that satisfies the following properties:
    \begin{enumerate}
    \item $A$ and $Q$ are equivalent.
        \item $q_{i,i}\vert q_{i+1,i+1}$, i.e. $q_{i,i}$ divides $q_{i+1,i+1}$   for all $i$.
        \item $q_{i,j}=0$ if $i\neq j$.
    \end{enumerate}
\end{definition}
Since the Smith normal form $Q$ is a diagonal matrix, we use the notation $q_i$ instead of $q_{i,i}$ to denote its $(i,i)$-th entry.
In case of a principal ideal domain, every matrix in $M_{m\times n}(R)$ has a unique Smith normal form up to units (\cite[Theorem 2.1]{stanley2016smith}). The diagonal entries $q_{i}$ in the Smith normal form of $A$ are known as the \emph{invariant factors} of $A$. We can find the Smith normal form using invertible elementary row and column operations on $A$ (\cite{NEWMAN1997367}). Another method of finding the Smith normal form is by computing the minors of $A$. Let $I_{k}(A)$ be the ideal generated by the set of all $k\times k$ minors of $A$. Note that $I_{k}(A)=0$ if $k>{\rm min}\{m,n\}$. In case of a principal ideal domain, $I_{k}(A)$ is generated by the gcd of all $k\times k$ minors of $A$. The following result connects the invariant factors with minors of $A$.
\begin{theorem}\label{minorsandinavrainats}\cite[Theorem 2.4]{stanley2016smith}
    Let $A\in M_{m\times n}(R)$ where $R$ is a principal ideal domain. Let $q_{1}\vert q_{2}\vert \dots$ be the invariant factors of $A$. We have the following:\begin{enumerate}
        \item  $ \langle q_{1}\cdots q_{k} \rangle =I_{k}(A)$ for all $1\leq k \leq {\rm min}\{m,n\}$;
       \item $q_{j}=0$ for all $j\geq {\rm min}\{m,n\}$.
    \end{enumerate}
\end{theorem}

In \cite[Theorem 1]{duffner2015interlacing}, Duffner and Silva established interlacing inequalities for the invariant factors of a matrix and its submatrices. Their result holds for matrices over \emph{elementary divisor duo rings}. Here, we restate it for $R=\mathbb{Z}$ which is an elementary divisor duo ring.
\begin{theorem}\label{dufsil}
Let $U\in M_{p\times q}(\mathbb{Z})$ be a submatrix of the matrix $W\in M_{m\times n}(\mathbb{Z})$. If $u_1\vert u_2 \dots$ and $w_1\vert w_2\dots $ are invariant factors of $U$ and $W$, respectively, then $w_i\vert u_i \vert w_{i+m+n-p-q}$ for all $i\geq 1$.
\end{theorem}
Let $G_1$ be a graph on the same vertex set as $G$. Consider the connected $k$-partition $\Pi=\{V_1,\dots,V_k\}$ of $V(G)$.
%Let $d_{1j},\dots,d_{nj}$ be the invariant factors (in Smith normal form) of $G_j$, and 
The set $\Pi$ is a partition of $V(G_1)$, but not necessarily a connected partition.
Let $b_{1,V_t}$ denote the sum of all columns in the Laplacian matrix of $G_1$ corresponding to the vertices in $V_t$. The set of vectors $\{b_{1,V_t}\mid 1\leq t \leq k\}$ generates the Laplacian lattice of the pseudo-pargraph $(G_1,\Pi)$.
%{\color{blue}
%\begin{definition}
 %   Let $G_1$ be a connected graph and $\Pi_1=\{V_1,\dots,V_m\}$ be a partition of $V(G_1)$. We call $(G_1,\Pi_1)$ a pseudo-pargraph if some of the induced subgraphs $G_1[V_i]$ are allowed to be disconnected.
%\end{definition}}
We say two pargraphs (pseudo-pargraphs) $(G,\Pi)$ and $(G_1,\Pi)$ have the same Laplacian lattice if $b_{1,V_t}=b_{V_t}$ for all $1\leq t \leq k$ where $\{b_{V_1},\dots,b_{V_k}\}$ is the generating set of $L_{\Pi}$ (Definition \ref{deflaplattice}).
\begin{theorem}\label{freeness of picard group}
Let $(G,\Pi)$ be a pargraph where $\Pi=\{V_1,\dots,V_k\} $ is a connected partition of $ V(G)$. Let $(G_1,\Pi), (G_2,\Pi),\dots,(G_m,\Pi)$ be a finite family of pseudo-pargraphs with $V(G_i)=V(G)$ for all $i$, and each $(G_i,\Pi)$ has the same Laplacian lattice as $(G,\Pi)$.
   If the total numbers of spanning trees in $G,G_1,\dots, G_m$ are relatively prime, then the Picard group of $(G,\Pi)$ is free.
\end{theorem}
\begin{proof}
     Consider the pargraph $(G,\Pi)$ with the connected partition $\Pi=\{V_1,\dots\newline, V_k\}$. We consider the case where $2\leq k \leq n-1$. 
     Let $M$ be the $n\times k$ matrix with column vectors $b_{V_1},\dots,b_{V_k}$.
    In order to prove the result, it is enough to show that the invariant factors of the matrix $M$ lie in the set $\{0,1\}$.

    Let $v_{i_1},\dots,v_{i_k}$ be the smallest index vertices in $V_1,\dots,V_k$, respectively. %Let $\Lambda_{G_l}$ be the Laplacian matrix of $G_l$. 
    Let $M_{l}$ be the matrix obtained from the Laplacian matrix $\Lambda_{G_l}$ of $G_l$ by applying the elementary column operations $b_{l,v_{i_t}}\rightarrow \sum_{{v_j}\in V_{t}}b_{l,v_j} $ for all $1\leq t \leq k$, where $b_{l,v_j}$ is the column corresponding to the vertex $v_j$ in $\Lambda_{G_l}$. Let $\tilde{M}$ be the matrix obtained from applying the above elementary column operations on $\Lambda_G$.
The matrix $M$ is a submatrix of  $\tilde{M}$ and $M_l$ for each $l$, as $(G_l,\Pi)$ and $ (G,\Pi)$ have the same Laplacian lattice.
    
    %Since both the pargraph $(G_1,\Pi), (G,\Pi)$ have the same Laplacian lattices, the matrix $M$ is a submatrix of $M_1$. 
    Since the above elementary column operations are invertible, the invariant factors of the matrix $M_l$  will be the same as those of $\Lambda_{G_l}${\cite{NEWMAN1997367}}.
    Note that the $(n-1)$-th invariant factor $d_{n-1}(G)$ is nonzero, and for some $l$, the corresponding invariant factor $d_{n-1}(G_l)$ of $G_l$ is also nonzero (from the matrix tree theorem, Theorem \ref{minorsandinavrainats}). 
    Let $a_1,\dots,a_k$ be the invariant factors of $M$. From Theorem \ref{dufsil}, $a_{k-1}$ divides $d_{(k-1+2n-(n+k))}(G)=d_{n-1}(G)$ and $d_{n-1}(G_l)$ for each $l$. Since the total numbers of spanning trees of $G,G_1,\dots, G_m$ are relatively prime, the matrix-tree theorem and Theorem \ref{minorsandinavrainats} imply that ${\rm{gcd}}\{d_{n-1}(G), d_{n-1}(G_l)\mid 1\leq l \leq m\}=1$. This proves that $a_{k-1}=1$.
    From the properties of invariant factors, $a_i \mid a_{k-1}$ for all $1 \leq i \leq k-1$. Hence, $a_i=1$ for all $1\leq i\leq k-1 $. Furthermore, ${\rm rank}(\Lambda_G)=n-1$ ensures that
    the matrix $M$ has rank $k-1$, and this proves $a_k=0$ (Theorem \ref{minorsandinavrainats}). 
    
    In case of $k=1$, the matrix $M$ is the zero vector in $\mathbb{Z}^{n}$, and hence, the Picard group is free.
\end{proof}

We now provide an intrinsic sufficient condition for the freeness of the Picard group of a special class of pargraphs. We say a pargraph $(G,\Pi)$ is \emph{simple} if the corresponding graph $G$ is simple, i.e. $G$ has no multiple edges and loops.
\begin{corollary}\label{picGfreethm}
    Let $(G,\Pi)$ be a simple pargraph where $\Pi=\{V_1,\dots,V_k\}$ is a connected partition of $V(G)$. If the graph obtained after removing the basic edges of $(G,\Pi)$ from $G$ is a forest, then the Picard group of $(G,\Pi)$ is free. 
\end{corollary}
\begin{proof}
    Let $E_{B}$ be the set of basic edges of $(G,\Pi)$. Let $G'$ be the graph obtained after removing the basic edges (non-relevant) of $(G,\Pi)$ from $G$, i.e. $V(G')=V(G)$ and $E(G')=E(G)\setminus E_{B}$. Let $\tilde{b}_{v}$ be the column in $\Lambda_{G'}$ that corresponds to the vertex $v\in V(G')$, and let $\tilde{b}_{V_i}=\sum_{v\in V_i}\tilde{b}_{v}$. Note that $\tilde{b}_{V_i}=b_{V_i}$ for all $1\leq i \leq k$. If $G'$ is a tree, then the Picard group of $(G,\Pi)$ is free (Theorem \ref{freeness of picard group}).

    Suppose that $G'$ is a forest. Let $T_i$ be a spanning tree of the subgraph $G[V_i]$ induced on $V_i$ by $G$. Consider the graph $\tilde{T}$ where $V(\tilde{T})=V(G')$ and $E(\tilde{T})=(\cup_{i=1}^{k}E(T_i))\cup E(G')$. Note that $b_{V_i}(\tilde{T})=b_{V_i}$ for all $1\leq i \leq k$ where $b_{V_i}(\tilde{T})$ is the vector obtained after adding all the columns in $\Lambda_{\tilde{T}}$ corresponding to the vertices in $V_i$. If $\tilde{T}$ is a tree, then we have already shown that the Picard group of $(G,\Pi)$ will be free. In case if $\tilde{T}$ is not a tree, then it will contain finitely many cycles. Note that each cycle in $\tilde{T}$ contains an edge that is a basic edge in $(G,\Pi)$. This holds as $G'$ is a forest. By successively removing one basic edge from each cycle in $\tilde{T}$, we obtain a tree, say $T$. Note that $V(T)=V(G)$ and $b_{V_i}(T)=b_{V_i}$ for all $i$. This proves that the Picard group of $(G,\Pi)$ is free.
\end{proof}
The converse of the above corollary does not hold in general, i.e. the Picard group of a pargraph $(G,\Pi)$ can be free even when the graph obtained by removing its basic edges from $G$ contains a cycle.

\begin{example}\label{exampofpicard} {\rm
    Consider the following graph $G$ along with pargraphs $(G,\Pi_1)$ and $(G,\Pi_2)$ where $\Pi_{1}=\{\{1\},\{2,5\},\{3\},\{4\},\{6\},\{7\}\}$ and $ \Pi_2= \{\{1\},\{2,5,7\},\newline\{3\},\{4\},\{6\}\}$ are the connected partitions of the vertex set of $G$.

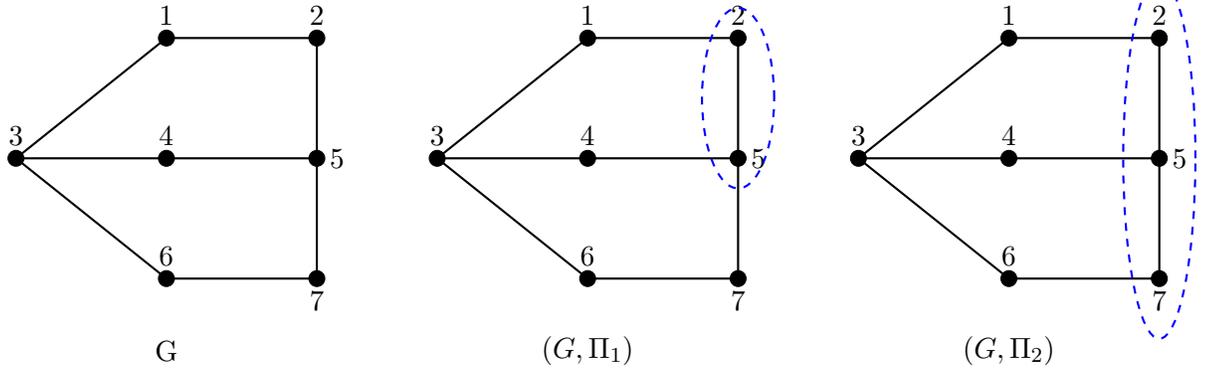
    \begin{figure}[htbp]
  \centering
  \resizebox{\textwidth}{!}{%
    \begin{tikzpicture}[scale=0.8, every node/.style={inner sep=1.5pt}]
      % =====================================================
      % === FIRST GRAPH (no ellipse) =========================
      % =====================================================
      \begin{scope}[xshift=0cm]
        \coordinate (c1) at (2.5,2);
        \coordinate (c2) at (5,2);
        \coordinate (c3) at (0,0);
        \coordinate (c4) at (2.5,0);
        \coordinate (c5) at (5,0);
        \coordinate (c6) at (2.5,-2);
        \coordinate (c7) at (5,-2);

        \foreach \a/\b in {1/2,2/5,3/1,3/4,3/6,4/5,6/7,7/5}{\draw[thick] (c\a)--(c\b);}

        \foreach \i/\pos in {1/above,2/above,3/above,4/above,5/right,6/above,7/below}{
          \node[circle,draw,fill,minimum size=6pt,label=\pos:{\small \i}] at (c\i) {};
        }
        \node at (2.5,-3.2) {\small $G$};
      \end{scope}

      % =====================================================
      % === SECOND GRAPH (ellipse around 2 and 5) ===========
      % =====================================================
      \begin{scope}[xshift=7cm]
        \coordinate (c1) at (2.5,2);
        \coordinate (c2) at (5,2);
        \coordinate (c3) at (0,0);
        \coordinate (c4) at (2.5,0);
        \coordinate (c5) at (5,0);
        \coordinate (c6) at (2.5,-2);
        \coordinate (c7) at (5,-2);

        \foreach \a/\b in {1/2,2/5,3/1,3/4,3/6,4/5,6/7,7/5}{\draw[thick] (c\a)--(c\b);}
        \foreach \i/\pos in {1/above,2/above,3/above,4/above,5/right,6/above,7/below}{
          \node[circle,draw,fill,minimum size=6pt,label=\pos:{\small \i}] at (c\i) {};
        }

        \draw[blue, thick, dashed] (5,1) ellipse (0.6cm and 1.5cm);
        \node at (2.5,-3.2) {\small $(G,\Pi_1)$};
      \end{scope}

      % =====================================================
      % === THIRD GRAPH (large ellipse around 5) ============
      % =====================================================
      \begin{scope}[xshift=14cm]
        \coordinate (c1) at (2.5,2);
        \coordinate (c2) at (5,2);
        \coordinate (c3) at (0,0);
        \coordinate (c4) at (2.5,0);
        \coordinate (c5) at (5,0);
        \coordinate (c6) at (2.5,-2);
        \coordinate (c7) at (5,-2);

        \foreach \a/\b in {1/2,2/5,3/1,3/4,3/6,4/5,6/7,7/5}{\draw[thick] (c\a)--(c\b);}
        \foreach \i/\pos in {1/above,2/above,3/above,4/above,5/right,6/above,7/below}{
          \node[circle,draw,fill,minimum size=6pt,label=\pos:{\small \i}] at (c\i) {};
        }

        \draw[blue, thick, dashed] (5,0) ellipse (0.6cm and 3cm);
        \node at (2.5,-3.2) {\small $(G,\Pi_2)$};
      \end{scope}
    \end{tikzpicture}%
  }
\end{figure}

The sets of basic edges of $(G,\Pi_1)$ and $(G,\Pi_2)$ are $\{\{2,5\}\}$ and $\{\{2,5\},\newline \{5,7\}\}$, respectively. In case of the pargraph $(G,\Pi_2)$, we obtain a tree if we remove the relevant edges of $(G,\Pi_2)$ from $G$. From Corollary \ref{picGfreethm}, the Picard group of $(G,\Pi_2)$ is free.

For the  pargraph $(G,\Pi_1)$, even though removing its relevant edges from $G$ does not yield a tree, the Picard group of $(G,\Pi_1)$ is free. This holds since the invariant factors of the matrix corresponding to the generating set $\{b_{\{2,5\}},b_{i}\mid i\in \{1,3,4,6,7\}\}$ of $L_{\Pi_1}$ are all $1$. We can also prove this using Theorem \ref{freeness of picard group}. Let $G_1=G-\{\{2,5\}\}$ be the graph obtained from $G$ by removing the edge $\{2,5\}$. Note that the pargraph $(G,\Pi_1)$ and the pseudo-pargraph $(G_1,\Pi_1)$ have the same Laplacian lattice, and the graphs $G$ and $G_1$ have $21$ and $5$ spanning trees, respectively. 
The freeness of the Picard group of $(G,\Pi_1)$ follows from Theorem \ref{freeness of picard group}.}\qed

%The sequences of invariants factors of $G$ and $G_1$ are $(1,1,1,1,1,21,0)$ and $(1,1,1,1,1,5,0)$, respectively. Since the gcd of the last nonzero invariants factors of $G$ and $G_1$ is $1$, and both pargraphs $(G,\Pi_1)$ and $(G_1,\Pi_1)$ have the same Laplacian lattice, the Picard group of $(G,\Pi_1)$ is free by Theorem \ref{freeness of picard group}.

%In case of the pargraph $(G,\Pi_1)$, the removal of the relevant edge of  $(G,\Pi_1)$ from $G$ doesn't provide a tree, but still the Picard group of $(G,\Pi_1)$ is free.

%The Picard groups of both pargraphs $(G,\Pi_1)$ and $(G,\Pi_2)$ are free because the invariant factors of the matrices corresponding to their Laplacian lattices are all $0$ or $1$.

%In case of the pargraph $(G,\Pi_2)$, $\{2,5\},\{5,7\}$ are the relevant edges of $(G,\Pi_2)$. If we remove these edges from the graph $G$, we obtain a tree. From Theorem \ref{picGfreethm}, the Picard group of $(G,\Pi_2)$ is free.}

\end{example}

We now establish a necessary condition for the freeness of the Picard group. For a subset $I\subseteq \{1,\dots,n\}$, we define the monomial ${\bf x}^{I}=\prod_{i\in I}x_i$.
Let $f(x_1,\dots,x_n)=\sum_{I\subseteq \{1,\dots,n\}}a_{I}{\bf x}^{I}\in \mathbb{Z}[x_1,\dots,x_n]$. Let $C_{f}$ be the content of $f$, i.e. the gcd of coefficients of $f$.
\begin{proposition}\label{conteProp}
    For any polynomial $f(x_1,\dots,x_n)=\sum_{I\subseteq \{1,\dots,n\}}a_{I}{\bf x}^{I}\in \mathbb{Z}[x_1,\dots,x_n]$, the content $C_f={\rm gcd}\{f(\sigma)\mid \sigma \in \{0,1\}^{n}\}$.
\end{proposition}
\begin{proof} Let $d={\rm gcd}\{f(\sigma)\mid \sigma \in \{0,1\}^{n}\}$.
Note that any $\sigma \in \{0,1\}^{n}$ can be written as $\chi_{T}$ for some $T\subseteq \{1,\dots,n\}$, where $\chi_T$ is the characteristic vector of $T$. The value of $f$ at $\sigma$ is $f(\sigma)=f(\chi_T)=\sum_{J\subseteq T}a_{J}\sigma^{J}= \sum_{J\subseteq T}a_{J}$. This shows $C_f$ divides $f(\sigma)$ for all $\sigma \in \{0,1\}^{n}$. Hence, $C_f$ divides $d$. 

From \cite[Proposition 3.7.1]{stanley2011enumerative} (M\"obius inversion formula), each $a_{I}$ can be expressed as a finite integral combination of $\{f(\sigma)\mid \sigma \in \{0,1\}^{n}\}$. This implies $d$ divides $a_{I}$ for all $I\subseteq \{1,\dots,n\}$. Hence, $d$ divides $C_f$ and this proves the result.
\end{proof}

 For a pargraph $(G,\Pi)$, let $P({\bf x})$ be its reduced Laplacian polynomial, and let $C_P$ be the content of $P({{\bf x}})$.

\begin{lemma}\label{freenessimplecontent}
    If the Picard group of a pargraph $(G,\Pi)$ is free, then $C_P=1$.
\end{lemma}
\begin{proof} Consider the pargraph $(G,\Pi)$ with the connected $k$-partition $\Pi=\{V_1,\dots,V_k\}$.
Let $M$ be the $n\times k$ matrix with column vectors $b_{V_1},\dots,b_{V_k}$. 
Since $\mathbb{Z}^{n}/L_{\Pi}$ is free, the gcd of $(k-1)\times (k-1)$ minors of $M$ is $1$ (Theorem \ref{minorsandinavrainats}). 

  Let $C_P>1$ and $p$ be a prime factor of $C_P$. The polynomial $P({\bf x})\equiv 0$ in $\mathbb{Z}_{p}[{\bf x}]$. Since $L({\bf x})$ is a weighted Laplacian matrix and $P({\bf x})\equiv 0$ in $\mathbb{Z}_{p}[{\bf x}]$, any $(n-1)\times (n-1)$ minor of $L({\bf x})$ is $0$ in $\mathbb{Z}_{p}[{\bf x}]$.
  This proves  ${\rm rank}(L({\bf{x}}))\leq n-2$ over the quotient field $\mathbb{Z}_{p}({\bf x})$. Consequently, the null space of $L(\mathbf{x})$ contains a non-zero vector $\mathbf{v} \in (\mathbb{Z}_{p}({\bf x}))^{n}$ that is not a scalar multiple of $(1, 1, \dots, 1)$. 

As ${\bf v}$ lies in the null space of $L(\mathbf{x})$, $L(\mathbf{x})\cdot {\bf v}=(L_R+L_B)\cdot {\bf v}={\bf 0}_n$ where ${\bf 0}_n$ is the column vector $(0,0,\dots,0)^t$. Let $i,j$ be two basic vertices that lie in the same main vertex, say $V_i$, and let $e=\{i,j\}$.
Taking the partial derivative of $(L_R+L_B)\cdot {\bf v}={\bf 0}_n$ with respect to $x_e(=x_{ij})$ yields
\begin{equation}\label{eq4}
    L({\bf x})\cdot \frac{\partial {\bf v}}{\partial x_e}+\frac{\partial L_B}{\partial x_e}\cdot {\bf v}={\bf 0}_n.
\end{equation}
Multiplying Equation \ref{eq4} from the left by ${\bf v}^t$, the transpose of ${\bf v}$, we obtain
\begin{equation}\label{eq5}
    {\bf v}^t \cdot L(X)\cdot \frac{\partial {\bf v}}{\partial x_e}+{\bf v}^t \cdot \frac{\partial L_B}{\partial x_e}\cdot {\bf v}= 0.
\end{equation}
Since $L({\bf x})$ is a symmetric matrix, ${\bf v}^t \cdot L({\bf x})=(L({\bf x})\cdot {\bf v})^{t}={\bf 0}$. Equation \ref{eq5} reduces to ${\bf v}^t \cdot \frac{\partial L_B}{\partial x_e}\cdot {\bf v}=0$, which yields $({\bf v}(i)-{\bf v}(j))^2=0$. This proves ${\bf v}(i)={\bf v}(j)$ for all $i,j\in V_i$, and hence, the entries of ${\bf v}$ are the same within each main vertex.

Let $U$ be the $n\times k$ matrix with columns $\chi_{V_1},\dots,\chi_{V_k}$ where $\chi_{V_i}$ is the characteristic vector of $V_i$. As ${\bf v}$ takes the same value on all basic vertices in a main vertex, and it is not a scalar multiple of $(1,1,\dots,1)$, there exists a non-zero column vector ${\bf w}\in (\mathbb{Z}_p({\bf x}))^k$ such that ${\bf v}=U\cdot {\bf w}$. Note that ${\bf w}$ is not a scalar multiple of $(1,1,\dots,1)$ and 
$M=L({\bf x)} \cdot U$. The vector $M\cdot {\bf w}=(L({\bf x})\cdot U)\cdot {\bf w}=L({\bf x})\cdot {\bf v}={\bf 0}_n$. This proves that ${\rm rank}(M)\leq k-2$ over the field $\mathbb{Z}_{p}({\bf x})$ and hence, over the field $\mathbb{Z}_{p}$. This implies $p$ divides all $(k-1)\times (k-1)$ minors of $M$. %Since the gcd of all $(k-1)\times (k-1)$ minors of $M$ is $1$, the prime $p$ divides $1$. 
This gives us a contradiction as gcd of all $(k-1)\times (k-1)$ minors of $M$ is $1$. Hence, $C_P$ does not have any prime factor and this proves the result.
\end{proof}
\subsection*{Proof of Theorem \ref{mainthm}}
\begin{proof} From Lemma \ref{freenessimplecontent}, $C_P=1$ if the Picard group of $(G,\Pi)$ is free. We now prove the converse part. Let $\tau(H)$ denotes the total number of spanning trees of a graph $H$. For a binary assignment ${\bf u}$ of the variables $x_{ij}$ in $\{0,1\}$, let $G({\bf u})$ be the graph with vertex set $V(G)$ and the Laplacian matrix $L{\bf(u)}$, where $L{\bf(u)}$ is the evaluation of the weighted Laplacian matrix at ${\bf x= u}$.
From the weighted matrix tree theorem \cite{klee2019linear} and Proposition \ref{conteProp}, we have
\begin{equation}
\begin{aligned}
    C_P &= \gcd\left\{ (L(\mathbf{u}))_{\mathrm{red}} \;\middle|\; \mathbf{u} \text{ is a binary assignment of the variables } x_{ij} \text{ in } \{0,1\} \right\} \\
        &= \gcd\left\{ \tau(G(\mathbf{u})) \;\middle|\; \mathbf{u} \text{ is a binary assignment of the variables } x_{ij} \text{ in } \{0,1\} \right\}.
\end{aligned}
\end{equation}
As $C_P=1$ and $\mathbb{Z}$ is a principal ideal domain, there exists a finite sequence of assignments ${\bf u_1,\dots,u_s}$ such that ${\rm gcd}\{\tau(G(\bf u_1)),\dots,\tau(G(\bf u_s))\}=1$. From Theorem \ref{freeness of picard group}, the Picard group $\mathbb{Z}^{n}/L_{\Pi}$ is free.
\end{proof}
\begin{example}\label{mainthmexam}{\rm Consider the pargraphs $(G,\Pi_1),(G,\Pi_2)$ defined in Example \ref{exampofpicard}. For $(G,\Pi_1)$, the reduced Laplacian polynomial ${\rm det}((L{(\bf{x}}))_{\rm red})=5+16 x_{25} \in \mathbb{Z}[x_{25}]$ and hence, the content of ${\rm det}((L{(\bf{x}}))_{\rm red})$ is ${\rm gcd}\{5,16\}=1$. From Theorem \ref{mainthm},
  the Picard group of $(G,\Pi_1)$ is free.

  For $(G,\Pi_2)$, ${\rm det}((L{(\bf{x}}))_{\rm red})= 12 x_{25}x_{27} + 12 x_{25}x_{57} + 12 x_{27}x_{57} + 4 x_{25} + 4 x_{27} + 4 x_{57} + 1
  \in \mathbb{Z}[x_{25},x_{27},x_{57}]$. Since the content of the reduced Laplacian polynomial of $(G,\Pi_2)$ is $1$, the Picard group of is free.\qed}
\end{example}
\begin{corollary}\label{corrpargrah} The Picard group of a pargraph $(G,\Pi)$ is free if and only if there exists a finite family of pseudo-pargraphs $(G_1,\Pi), (G_2,\Pi),\dots,(G_m,\Pi)$ such that $V(G_i)=V(G)$ for all $i$, each $(G_i,\Pi)$ has the same Laplacian lattice as $(G,\Pi)$ and the number of spanning trees in $G,G_1,\dots, G_m$ are relatively prime.
\end{corollary}
\begin{proof}
    The proof follows from Theorem \ref{freeness of picard group} and the proof of the converse of Theorem \ref{mainthm}.
\end{proof}

Recall that the toppling ideal of a pargraph $(G,\Pi)$ is the lattice ideal associated with the Laplacian lattice $L_{\Pi}$ of the pargraph. Note that the Laplacian lattice $L_{\Pi}$ is a sublattice of the root lattice $A_{n-1}=\{(a_1,\dots,a_n)\in \mathbb{Z}^{n} \mid a_1+\dots +a_n=0\}$. As a consequence of this, the toppling ideal $I_{\Pi}$ is a homogeneous ideal in the polynomial ring $R_n$. 

\begin{corollary}\label{corrtoptoricideal}
The toppling ideal $I_{\Pi}$ of a pargraph $(G,\Pi)$ is toric if and only if the content of the reduced Laplacian polynomial of $(G,\Pi)$ is $1$.
\end{corollary}
\begin{proof}
    The proof follows from Theorem \ref{mainthm}, and \cite[Theorem 7.4]{miller2005combinatorial}.
\end{proof}

Note that if $G=C_n$ is the cycle graph on $n$-vertices and $\Pi$ is connected $k$-partition with $
|V_i|\geq 2$ for some $V_i\in \Pi$, then the toppling ideal of $(C_n,\Pi)$ is a prime ideal (Corollary \ref{picGfreethm}). This provides another proof of the toric structure of the toppling ideal of  $(C_n,\Pi)$ \cite[Corollary 3.9]{karki2024rationalnormalcurveschip}.

\subsection{Gr\"obner Basis of the Toppling Ideal}\label{secgboftop}
 %The toppling ideal of the pargraph $(G,\Pi)$ is the lattice ideal associated with the Laplacian lattice $L_{\Pi}$ the pargraph. Note that the Laplacian lattice $L_{\Pi}$ is a sublattice of the root lattice $A_{n-1}=\{(a_1,\dots,a_n)\in \mathbb{Z}^{n} \mid a_1+\dots +a_n=0\}$. As a consequence of this, the toppling ideal $I_{\Pi}$ is a homogeneous ideal in the polynomial ring $R_n$. 
 
 In this subsection, we first construct a generating set, and then a Gr\"obner basis for the toppling ideal in terms of certain subsets of $V(G)$.
For a parset $S$ of $V(G)$, we define the monomial ${\bf x}^{S\rightarrow \overline{S}}=\prod_{i\in S}x_i^{{\rm outdeg}_{S}(i)}$. %where $d_i$ is the number of edges between the vertex $i$ and the vertices in $\overline{S}=V(G)\setminus S$. 
We use the notation $S\rightarrow \overline{S}$ to denote the exponent vector of the monomial ${\bf x}^{S\rightarrow \overline{S}}$.
%Let $S$ be a parset of $V(G)$ not containing $V_k$. We define the vector $S\rightarrow \bar{S}=$
\begin{lemma}\label{generating lemma} Let $(G,\Pi)$ be a pargraph where $\Pi=\{V_1,\dots,V_k\}$ is a connected partition of $V(G)$. The toppling ideal $I_{\Pi}$ is generated by the set $\{ {\bf x}^{S\rightarrow \bar{S}}-{\bf x}^{ \bar{S}\rightarrow S} \mid S \text{ is a parset of }V(G)\text{ not containing }V_k
\}.
$ 
\end{lemma}
\begin{proof}
    Let $S$ be a parset of $V(G)$ not containing $V_k$. Note that $S\rightarrow  \bar{S}, \bar{S} \rightarrow  S \in  \mathbb{N}^{n}$ and $ (S\rightarrow  \bar{S}) - (\bar{S} \rightarrow  S) = \sum_{i\in S} b_i \in L_{\Pi}$. Hence, ${\bf x}^{ S\rightarrow  \bar{S}}-{\bf x}^{ \bar{S} \rightarrow  S}\in I_{\Pi}$ and  $\langle {\bf x}^{ S\rightarrow  \bar{S}}-{\bf x}^{ \bar{S} \rightarrow  S} \mid S \text{ is a parset of $V(G)$ not containing } V_k \rangle \subseteq I_{\Pi}$.

    Let ${\bf x^{u}- x^{v}}\in I_{\Pi}$ where ${\bf u},{\bf v}\in \mathbb{N}^{n}$ and ${\bf u}-{\bf v}\in L_{\Pi}$.   Let $S$ be a parset of $V(G)$ such that ${\bf x^u}$ is divisible by ${\bf x}^{ S\rightarrow  \bar{S}}$. The monomial ${\bf x^{u}}={\bf x^{m}}({\bf x}^{ S\rightarrow  \bar{S}}-{\bf x}^{ \bar{S} \rightarrow  S})+{\bf x^r}$ where ${\bf x^r}={\bf x}^{{\bf u}- ((S\rightarrow  \bar{S}) - (\bar{S} \rightarrow  S))}$ is the remainder obtained after dividing ${\bf x^{u}}$ by ${\bf x}^{ S\rightarrow  \bar{S}}-{\bf x}^{ \bar{S} \rightarrow  S}$. This implies that for the divisor ${\bf u}$ on the pargraph $(G,\Pi)$, the parset $S$ is legal to fire, and as a consequence, the new divisor we obtain is ${\bf u}-((S \rightarrow \bar{S})-(\bar{S}\rightarrow S))$. From Lemma \ref{effdiv}, there exists a finite sequence of legal chip-firing moves from the parsets not containing $V_k$ that reduces the divisor ${\bf u}$ to a $V_k$-reduced divisor on $(G,\Pi)$. %Hence,
%${\bf x^{u}}=\sum_{j=1}^{l} {\bf x^{m_{1j}}} ({\bf x}^{ S_{1j} \rightarrow  \bar{S}_{1j}}-{\bf x}^{\bar{S}_{1j} \rightarrow  S_{1j}})+ {\bf x}^{D_1}$ where $S_{11},\dots ,S_{1l}$ is the sequence of parsets and $D_1 \in \mathbb{N}^{n}$ is the $V_k$-reduced divisor. Similiarly, ${\bf x^{v}}=\sum_{j=1}^{p} {\bf x^{m_{2j}}} ({\bf x}^{ S_{2j} \rightarrow  \bar{S}_{2j}}-{\bf x}^{\bar{S}_{2j} \rightarrow  S_{2j}})+ {\bf x}^{D_2}$  where $S_{21},\dots ,S_{2p}$ is a sequence of parset not containing $V_k$ and $D_2 \in \mathbb{N}^{n}$ is a $V_k$-reduced divisor.
Hence,
${\bf x^{u}}=\sum_{j=1}^{l} {\bf x^{m_{1j}}} ({\bf x}^{ S_{1j} \rightarrow  \bar{S}_{1j}}-{\bf x}^{\bar{S}_{1j} \rightarrow  S_{1j}})+ {\bf x}^{D_1}$, and similarly, ${\bf x^{v}}=\sum_{j=1}^{p} {\bf x^{m_{2j}}} ({\bf x}^{ S_{2j} \rightarrow  \bar{S}_{2j}}-{\bf x}^{\bar{S}_{2j} \rightarrow  S_{2j}})+ {\bf x}^{D_2}$  where $S_{11},\dots ,S_{1l}$ and $S_{21},\dots ,S_{2p}$ are sequences of parsets not containing $V_k$, and $D_1,D_2 \in \mathbb{N}^{n}$ are $V_k$-reduced divisors. Since ${\bf u}$ and ${\bf v}$ lie in the same chip-firing equivalence class, both are equivalent to the same $V_k$-reduced divisor (Lemma \ref{effdiv}). Hence, $D_1=D_2$ and ${\bf x^{u}}-{\bf x^{v}} \in \langle {\bf x}^{ S\rightarrow  \bar{S}}-{\bf x}^{ \bar{S} \rightarrow  S} \mid S \text{ is a parset of } V(G) \text{ not containing } V_k \rangle$.
\end{proof}

Let $\lambda\in \mathbb{N}^{n}$ be a weight vector such that $b_{V_i}\cdot \lambda > 0$ for all $i\neq k$. An integral solution of the equation $\Lambda_{G} \tilde{\lambda}=y$ where $y=(1,\dots,1,-(n-1))$ provides such a weight vector. This weight vector $\tilde{\lambda}$ arises as a potential vector $b_q$ in \cite{ChipFirePotential}. Consider a rooted spanning tree $(T,n)$ of $(G,\Pi)$. We define a monomial order $\bar{rev}$ on $R_n=\mathbb{K}[x_1,\dots,x_n]$ as follows. The monomial order $\bar{rev}$ is the reverse lexicographic order induced by the ordering of variables $<'$ where $x_i<'x_j$ if the vertex $j$ is a descendant of the vertex $i$ in $(T,n)$. The monomial order $\bar{rev}$ is known as the \emph{spanning tree order} on $R_n$ induced by $(T,n)$.
Let $m_{rev}$ be the weighted monomial order defined as ${\rm in}_{m_{rev}}(f)={\rm in}_{\bar{rev}}({\rm in}_{\lambda}(f))$ where $f$ is a polynomial in $R_n$.

Note that for a connected parset $S$ of $V(G)$ not containing $V_k$, we have $((S \rightarrow \overline{S}) - (\overline{S} \rightarrow S)) \cdot \lambda > 0$, since $S = \bigcup_{i \in \tilde{S}} V_i$ for some $\tilde{S} \subseteq \{1, \dots, k-1\}$ and $b_{V_i} \cdot \lambda > 0$ for all $i$ from $1$ to $k-1$. Hence, ${\rm in}_{m_{rev}}({\bf x}^{S \rightarrow \overline{S}}-  {\bf x}^{\overline{S} \rightarrow S}      )={\bf x}^{S \rightarrow \overline{S}}$ for all connected parsets $S$ of $V(G)$ not containing $V_k$.
\subsection*{Proof of Theorem \ref{gboftop}}
\begin{proof}
    From Lemma \ref{generating lemma}, the toppling ideal is generated by the set $M=\{ {\bf x}^{S\rightarrow \bar{S}}-{\bf x}^{ \bar{S}\rightarrow S} \mid S \text{ is a parset of }V(G)\text{ not containing }V_k
\}$. From  \cite[Theorem 3.7]{karki2024rationalnormalcurveschip}, for any two parsets $S_1$ and $S_2$ of $V(G)$ not containing $V_k$, 
the $S$-polynomial of ${\bf x}^{S_1\rightarrow \overline{S}_1}-{\bf x}^{ \overline{S}_1\rightarrow S_1}$ and ${\bf x}^{S_2\rightarrow \overline{S}_2}-{\bf x}^{ \overline{S}_2\rightarrow S_2}$ reduces to $0$ (as per the division algorithm \cite[Theorem 2.2.1]{herzog2011monomial}) with respect to $B(S_1\setminus S_2)$ and $B(S_2\setminus S_1)$. Hence, by Buchberger's criterion, the set $M$ forms a Gr\"obner basis of $I_{\Pi}$.

 We now prove that the initial term ${\rm in}_{m_{rev}}({\bf x}^{S\rightarrow \overline{S}}-{\bf x}^{ \overline{S}\rightarrow S})={\bf x}^{S\rightarrow \overline{S}}$ of a parset $S$ not containing $V_k$ is divisible by an initial term ${\bf x}^{U\rightarrow \overline{U}}$, where $U$ and $\overline{U}$ are connected parset of $V(G)$ and $V_k\subseteq \overline{U}$. 
 Suppose the parset $S$ is not connected and $U$ is a connected component of $S$. Note that, $(U\rightarrow \overline{U})(i)\leq (S\rightarrow \overline{S})(i)$ for all $i$ from $1$ to $n$, and $\overline{U}$ is connected. Hence, for any connected component $U$ of $S$, ${\bf x}^{U\rightarrow \overline{U}}$ divides ${\bf x}^{S\rightarrow \overline{S}}$.

 Suppose the parset $\overline{S}$ is not connected. Let ${U}_1$ be the connected component of $\overline{S}$ that contains $V_k$. Note that $U_2=\overline{U}_1$ is a connected parset that does not contain $V_k$ and $(U_2\rightarrow \overline{U}_2)(i)\leq (S\rightarrow \overline{S})(i)$ for all $i$ from $1$ to $n$. This implies ${\bf x}^{U_2\rightarrow \overline{U}_2}$ divides ${\bf x}^{S\rightarrow \overline{S}}$ and $U_2,\overline{U}_2=U_1$ are both connected parsets.
From \cite[Definition 2.1.5, Theorem 2.1.8]{herzog2011monomial},  the set$\{{\bf x}^{S \rightarrow \overline{S}}-  {\bf x}^{\overline{S} \rightarrow S} \mid S  \text{ and }\overline{S} \text{ are connected parsets of } V(G)\text{ and } V_k\subseteq \overline{S}\}$ forms a Gr\"obner basis of the toppling ideal $I_{\Pi}$ with respect to the monomial order $m_{rev}$.
\end{proof}

In case of a graph $G$, the initial ideal $M_G=\langle {\bf x}^{S \rightarrow \overline{S}} \mid  S\subset V(G)\setminus {n} \text{ and } G[S], G[\overline{S}] \text{ are connected} \rangle$ is called the $G$-parking function ideal of $G$. This is due to the connections between the standard monomials of $M_G$ with the \emph{$G$-parking functions}{\cite{ds},\cite{manjunath2013monomials}}. We extend this notion of the $G$-parking function ideal of a graph to a pargraph.

\begin{definition}
    The $G$-parking function ideal of a pargraph $(G,\Pi)$ is the (initial) monomial ideal $M_{\Pi}=\langle {\bf x}^{S\rightarrow \overline{S}}\mid S  \text{ and }\overline{S} \text{ are connected parsets of } V(G)\newline \text{ and } V_k\subseteq \overline{S}\rangle$

    %S \text{ is a parset of }V(G) \text{ not containing }\newline V_k\rangle$.

\end{definition}

\begin{remark}{\rm In case of the pargraph $(G,\Pi)$, where $\Pi=\{V_1,\dots, V_n\}$ and $V_i=\{i\}$ for all $1\leq i \leq n$, the notion of the $G$-parking function ideal of the pargraph coincides with the notion of $G$-parking function ideal of the graph $G$.}
    
\end{remark}
\begin{proposition}\label{prop3.6}
    The $G$-parking function ideal $M_{\Pi}$ of a pargraph $(G,\Pi)$ is minimally generated by the set $\{{\bf x}^{S \rightarrow \overline{S}}\mid S  \text{ and }\overline{S} \text{ are connected parsets of } V(G)\newline \text{ and } V_k\subseteq \overline{S}\}$.
\end{proposition}
\begin{proof}
    From Theorem \ref{gboftop}, the set $\mathcal{M}=\{{\bf x}^{S \rightarrow \overline{S}}\mid S  \text{ and }\overline{S} \text{ are connected parsets}\newline \text{ of } V(G)\text{ and } V_k\subseteq \overline{S}\}$ generates the $G$-parking function ideal $M_{\Pi}$. In order to prove the result, we show that if we take any two distinct monomials from $\mathcal{M}$, neither of them divides the other.
    Let ${\bf x}^{S_1 \rightarrow \overline{S}_1}$ and ${\bf x}^{S_2 \rightarrow \overline{S}_2}$ be two distinct monomials in $\mathcal{M}$. Without loss of generality, assume that ${\bf x}^{S_1 \rightarrow \overline{S}_1}$ is divisible by ${\bf x}^{S_2 \rightarrow \overline{S}_2}$.
Suppose that $S_2\subset S_1$. Since ${\bf x}^{S_2 \rightarrow \overline{S}_2}$ divides ${\bf x}^{S_1 \rightarrow \overline{S}_1}$, ${\rm outdeg}_{S_2}(i)\leq {\rm outdeg}_{S_1}(i)$ for all $i\in S_2$. This implies that there are no edges between $S_2$ and $S_1 \setminus S_2$, which is a contradiction since $S_1$ is a connected parset. 

Suppose that $S_2 \not\subset S_1$. Since ${\bf x}^{S_2 \rightarrow \overline{S}_2}$ divides ${\bf x}^{S_1 \rightarrow \overline{S}_1}$, ${\rm outdeg}_{S_2}(i)=0$ for all $i\in S_2\setminus S_1$. This implies that there are no edges between $S_2\setminus S_1$ and $\overline{S}_1 \setminus S_2$, which is a contradiction as $\overline{S}_1$ is a connected parset. This proves that the set $\mathcal{M}$ minimally generates $M_{\Pi}$. 
\end{proof}
\subsection{Cellular Free Resolutions of $G$-Parking Function Ideals}\label{cellresofparkingidelas}
In this subsection, we construct a (minimal) \emph{cellular free resolution} for the $G$-parking function ideal of a pargraph $(G,\Pi)$. For this, we define the notion of a \emph{graphical hyperplane arrangement} of a pargraph.  The collection of bounded cells in this arrangement gives us a \emph{polyhedral cell complex} which produces the desired free resolution. This extends the work of the author and Manjunath on parcycles to pargraphs \cite{karki2024rationalnormalcurveschip}. For graphs, the construction of cellular free resolutions for $G$-parking function ideals using graphical arrangements was first introduced by Dochtermann and Sanyal \cite{ds}.

We briefly discuss the construction of a minimal cellular free resolution of the $G$-parking function ideal $M_G$, as defined for a graph $G$ in \cite{ds}. For a given graph $G$ on the vertex set $\{1,\dots,n\}$, we associate an affine space $U_G=\{{\bf p}=(p_1,\dots,p_n) \in \mathbb{R}^{n} \mid p_n=0,p_1+\dots+p_{n-1}=1\}$  and a hyperplane arrangement $\mathcal{A}_G=\{h_{ij}\mid \{i,j\} \text{ is an edge of } G\}$ where $h_{ij}=\{{\bf p}\in \mathbb{R}^{n} \mid p_i=p_j \}$. By restricting the arrangement $\mathcal{A}_G$ to $U_G$, we obtain a polyhedral cell complex, denoted by $\mathcal{B}_G$, that consists of all the bounded cells resulting from this restriction. A suitable labeling of the vertices ($0$-dimensional cells) of $\mathcal{B}_G$ with the minimal generators of $M_G$ provides us a labelled polyhedral cell complex. This labelled polyhedral cell complex produces a minimal cellular free resolution of $M_G$.

 %Let $G$ be a connected multigraph on the vertex set $\{1,\dots,n\}$. Let $(G,\Pi)$ be a pargraph where $\Pi=\{V_1,\dots,V_k\}$ is a connected partition and $n\in V_k$. 

 Given a pargraph $(G,\Pi)$ where $\Pi=\{V_1,\dots,V_k\}$ and $n\in V_k$, we define the affine subspace $U_{\Pi}=\{{\bf p}\in \mathbb{R}^{n} \mid p_n=0,p_1+\dots+p_{n-1}=1, p_{i_1}=p_{i_2} \text{ if }\{i_1,i_2\}\subseteq V_l \text{ for some } l \text{ from }1 \text{ to } k\} \subseteq \mathbb{R}^{n}$. %Consider the pargraph $(G,\Pi)$ where $\Pi=\{V_1,\dots,V_k\}$ and $n\in V_k$. Consider the affine subspace $U_{\Pi}=\{{\bf x}\in \mathbb{R}^{n} \mid x_n=0,x_1+\dots+x_{n-1}=1, x_{i_1}=x_{i_2} \text{ if }\{i_1,i_2\}\subseteq V_l \text{ for some } l\}$ of $\mathbb{R}^{n}$.
 For each relevant edge $\{i,j\}$ of $(G,\Pi)$, we introduce a hyperplane $h_{ij}=\{{\bf p}=(p_1,\dots,p_n)\in \mathbb{R}^{n} \mid p_i=p_j \}$. 
 We call the arrangement $\mathcal{A}_{\Pi}=\{h_{ij}\mid \{i,j\} \text{ is a relevant edge of } (G,\Pi) \}$ of hyperplanes in $\mathbb{R}^{n}$ as the \emph{graphical hyperplane arrangement} of $(G,\Pi)$. In order to construct the desired polyhedral cell complex, we restrict the arrangement $\mathcal{A}_{\Pi}$ on the affine subspace $U_{\Pi}$. The restricted hyperplane arrangement $\tilde{\mathcal{A}}_{\Pi}=\{ h_{ij} \cap U_{\Pi} \mid  \{i,j\} \text{ is a relevant edge of } (G,\Pi)\}$ is an \emph{essential} hyperplane arrangement in the sense of \cite[Subsection 1.1]{stanley2004introduction}. The arrangement $\tilde{\mathcal{A}}_{\Pi}$ partitions the space $U_{\Pi}$ into polyhedra of different dimensions, which are called \emph{cells}. Let ${\mathcal{B}}_{\Pi}$ be the collection of all bounded cells in $U_{\Pi}$. 
 Let $E_B$ be the set of basic edges of $(G,\Pi)$. Note that $U_{\Pi}=U_G \bigcap_{{\{i,j\}\in E_{B}}}h_{ij}$ and $\tilde{\mathcal{A}}_{\Pi}$ is the restriction of ${\mathcal{A}}_G$ on $U_{\Pi}$. This implies the elements of ${\mathcal B}_{\Pi}$ are also elements of $\mathcal{B}_G$ and hence, $\mathcal{B}_{\Pi}$ is a subcomplex of the polyhedral complex $\mathcal{B}_G$.

 From \cite[Proposition 3, Corollary 2]{ds}, the $0$-dimensional cells of $\mathcal{B}_{\Pi}$ are of the form $\chi_{J}/|J|$ where $J\subseteq \{1,\dots,n-1\}$ such that the induced subgraphs $G[J]$ and $G[V(G)\setminus J]$ are connected, and $\chi_{J}$ is the characteristic vector of the subset $J$. Since the $0$-dimensional cells of $\mathcal{B}_{\Pi}$ lie on $U_{\Pi}$, every $0$-dimensional cell of $\mathcal{B}_{\Pi}$ will be of the form $\chi_{J}/|J|$ where $J\subseteq \{1,\dots,n-1\}$ and both $J$ and $\overline{J}= V(G)\setminus J$ are connected parsets of $V(G)$. Let $|\mathcal{B}_{\Pi}|$ be the point set of $\mathcal{B}_{\Pi}$. To determine whether a point $p\in U_{\Pi}$ is in $|\mathcal{B}_{\Pi}|$, we construct a directed graph $G/p$ and use \cite[Proposition 6.1]{ds}. We first partition the vertex set of $G$ according to the equivalence relation where two vertices $i$ and $j$ are related if there exists a path $i=i_0i_1\dots i_{l}=j$ in $G$ such that $p(i_0)=p(i_1)=\dots=p(i_l)$. Let $W_1,W_2,\dots,W_m$ be the equivalence classes. Note that each induced subgraph $G[W_i]$ is connected. The underlying undirected graph of $G/p$ is obtained by contracting each $G[W_i]$ to a single vertex. An edge in $G/p$ connecting $G[W_i]$ to $G[W_j]$ has the orientation $(G[W_i],G[W_j])$ if $p(i_s)>p(i_t)$ for some $i_s\in W_i$ and $i_t\in W_j$. From \cite[Proposition 6.1]{ds}, a point $p\in U_\Pi$ lies in $|\mathcal{B}_{\Pi}|$ if $G/p$ has a unique sink at the vertex class containing $n$.

% We now associate monomials to the vertices of $\mathcal{B}_{\Pi}$.
For a  vertex $p=\chi_{J}/|J|$ of $\mathcal{B}_{\Pi}$, we associate it with the monomial ${\bf m}_{p}= {\bf x}^{J\rightarrow \overline{J}}= \prod_{i\in J}x_i^{d_i}$, where $d_i={\rm outdeg}_{J}(i)$. Using the assignment of monomials to vertices above, for each bounded cell $\tilde{B} \in \mathcal{B}_{\Pi}$ we define the associated monomial $\mathbf{m}_{\tilde{B}}$ to be the least common multiple of $\{\mathbf{m}_u \mid u \text{ is a vertex of } \tilde{B}\}$.
%Using the above assignment of monomials to vertices, we define for each bounded cell $\tilde{B}\in \mathcal{B}_{\Pi}$ the associated monomial ${\bf m}_{\tilde{B}}$ to be the least common multiple of $\{{\bf m}_{u}\mid u \text{ is a vertex of } \tilde{B}\}$.
%Using the above assignment of monomials to vertices, we associate to each bounded cell $\tilde{B}\in \mathcal{B}_{\Pi}$ the monomial ${\bf m}_{\tilde{B}}$, defined as the least common multiple of the monomials $\{{\bf m}_{u}\mid u \text{ is a vertex of } \tilde{B}\}$. 
This assignment of monomials turns $\mathcal{B}_{\Pi}$ into a labelled polyhedral cell complex. 
 
\begin{lemma}\label{starconvexity lemma}  Let $(G,\Pi)$ be a pargraph and $\mathcal{B}_{\Pi}$ be the corresponding labelled polyhedral cell complex. The set $|\mathcal{B}_{\Pi}|_{\leq \sigma}$ is star-convex for every $\sigma \in \mathbb{N}^{n-1}$.
    \end{lemma}
\begin{proof}
Let $p_1,\dots,p_{m}$ be the $0$-dimensional cells in $|\mathcal{B}_{\Pi}|_{\leq \sigma}$. We know each $p_j=\chi_{I_j}/{|I_j|}$ where $I_j\subseteq \{1,\dots,n-1\}$ and both $I_j$ and $\overline{I_j}$ are connected parsets of $V(G)$. Let $a_{p_j}$ be the exponent vector of the monomial associated with $p_j$, i.e. $a_{p_j}=I_{j}\rightarrow \overline{I_j}$.
Consider the set  $S=\bigcup_{j=1}^{m} {\rm Supp}(a_{p_j}) \subset V(G)$. Let $K$ be the set of vertices of the connected component of $G[V(G)\setminus S]$ that contains $n$. Let $W=V(G)\setminus K$ and $q=\chi_{W}/|W|$. Note that the directed graph $G/q$ has a unique sink at the vertex class containing $n$. From \cite[Proposition 6.1]{ds}, the vertex $q$ lies in a bounded cell of $B_G$.

We now prove that $q\in U_{\Pi}$. For this, it is enough to show that for each $j$, either $V_j\subseteq W$ or $V_j\subseteq K$. Suppose that there exists a main vertex $V_i\in \Pi$ that is neither a subset of $W$  nor a subset of $K$. There exist vertices $i_{1}, i_{2}\in V_i$ such that $\{i_{1}, i_{2}\}\in E(G)$, with $i_{1}\in W$ and $i_{2}\in K$. 
Note that the vertex $i_1$ must also belong to $S$. This holds because if $i_1\notin S$, then the edge $\{i_{1}, i_{2}\}\in E(G)$ and the condition $i_{2}\in K$ would imply $i_1\in K$ by the construction of $K$, a contradiction. Since $i_1\in S (=\bigcup_{j=1}^{m} {\rm Supp}(a_{p_j}))$ and every $0$-dimensional cell of $\mathcal{B}_{\Pi}$ is of the form $\chi_{J}/|J|$, there exists a connected parset $J$ of $V(G)$ such that $i_1\in J$. From the construction of $W$, ${\rm Supp}(p_j)\subseteq W$ for all $j$ from $1$ to $m$. This implies $i_2\in \overline{J}, i_1\in J$ and it  provides us a contradiction as $J$ and $\overline{J}$ are connected parsets of $(G,\Pi)$ and $\{i_1,i_2\}\subseteq V_i$.
Therefore, our assumption is false, which implies $q \in U_{\Pi}$ and, in particular, $q \in B_{\Pi}$.%Hence, our assumption is wrong, and thus $q \in U_{\Pi}$ and in particular $q \in B_{\Pi}$.

Let $C_q$ be the inclusion-minimal bounded cell on $U_{\Pi}$ that contains $q$. From \cite[Proposition 6.3]{ds}, the exponent vector of the monomial associated with $C_q$ is $a_q=(W\rightarrow \overline{W})$. In order to show $q\in |\mathcal{B}_{\Pi}|_{\leq \sigma}$, it is enough to $a_q\leq \sigma$. Let $i_s\in W$ such that $a_q(i_s)=(W\rightarrow \overline{W})(i_s)>0$. There exists $i_t\in K$ such that $\{i_s,i_t\}\in E(G)$. As noted earlier, if $\{i_s,i_t\}\in E(G)$ and $i_s\in W, i_t\in K$, then $i_s\in S$. Hence, $i_s\in {\rm Supp}(a_{p_l}) \subseteq I_l$ for some $l$. Since $I_{l}\subseteq W$, $a_q(i_s)(=d_{W}(i_s))\leq a_{p_l}(i_s) (=d_{I_l}(i_s))\leq \sigma(i_s)$. This implies $q\in |\mathcal{B}_{\Pi}|_{\leq \sigma}$.

We now show that $q$ is a star-center for $|\mathcal{B}_{\Pi}|_{\leq \sigma}$. Let $t\in |\mathcal{B}_{\Pi}|_{\leq \sigma}$ be an arbitrary point. Let $[q,t]$ be the line segment connecting the points $q$ and $t$. Let $t(i)>t(j)$ for some relevant edge $\{i,j\}$ of $(G,\Pi)$. From \cite[Proposition 6.3]{ds}, $a_{C_t}(i)>0$ where $C_t$ is the inclusion minimal bounded cell in $\mathcal{B}_{\Pi}$ containing $t$, and $a_{C_t}$ is the exponent vector of the monomial associated with it. This implies $i\in S\subseteq W$ and thus $q(i)=1/|W|\geq q(j)$. This shows that no hyperplanes corresponding to the relevant edges of $(G,\Pi)$ separate the points $t$ and $q$ on $U_{\Pi}$. Thus, the open line segment $(t,q)$ is contained in some cell, say  $C_{(t,q)}$, of $\tilde{\mathcal{A}}_{\Pi}$.

 Next, we prove that $C_{(t,q)}$ is a bounded cell and lies in $|\mathcal{B}_{\Pi}|_{\leq \sigma}$.
Let $g$ be a point in the open line segment $(t,q)\subseteq {\rm reint}(C_{(t,q)})$.
From \cite[Proposition 6.1]{ds}, $C_{(t,q)}\in {\mathcal{B}_{\Pi}}$ if $G/g$ has a unique sink at the vertex class containing $n$.
 Suppose there exists another sink in the vertex class containing a vertex $i$, which is distinct from the vertex class containing $n$. Since $t\in |\mathcal{B}_{\Pi}|_{\leq \sigma}$, there exists a path $i=i_0i_1\cdots i_{m}=n$ such that $t(i_{l-1})\geq t(i_l)$ for all $1\leq l \leq m$. %As $G/g$ has a sink at the vertex class containing $i$,
 From the assumption, the path is not weakly decreasing for $g$. This implies there exists an index $j$ such that $g(i_{j-1})<g(i_j)$, and thus $g(i_j)>0$. Note that ${\rm Supp}(g)={\rm Supp}(t)\cup {\rm Supp}(q)$ and by construction ${\rm Supp}(t)\subseteq {\rm Supp}(q)$. This implies $q(i_j)=1/|W|$ and the path is not weakly decreasing for $q$. We have $g\in (t,q)$ and $(t,q)\subseteq \{p\in U_{\Pi}\mid p(i_{j-1})\geq p(i_j)\}$, and thus  $g(i_{j-1})\geq g(i_j)$. This contradicts the condition that  $g(i_{j-1})<g(i_j)$. This proves $G/g$ has a unique sink at the vertex class containing $n$.
 
 Let $a_{g}$ be the exponent vector of the monomial associated with $C_{(t,q)}$. From \cite[Proposition 4]{ds}, $a_g(i)=\#\{\{i,j\}\in E(G) \mid g(i)>g(j)\}$. If $g(i)=0$, then $a_g(i)=0$. We now assume that $g(i)>0$.
 As $g\in (t,q)$, $g(i)>g(j)$ if and only if $wt(i)+(1-w)q(i)> wt(j)+(1-w)q(j)$ where $w\in (0,1)$. If $t(i)=0$, then $q(i)>q(j)$ whenever $g(i)>g(j)$. This implies $a_g(i)\leq a_q(i)\leq \sigma(i)$. We now consider the case $t(i)> 0$. As ${\rm Supp}(t)\subseteq {\rm Supp}(q)$, $i\in W$ and $q(i)>0$. 
 Furthermore, whenever $t(i)> 0$, the relation $g(i)>g(j)$ implies $t(i)>t(j)$.
 This implies $a_g(i)\leq a_{C_t}(i)\leq \sigma(i)$ and hence, $C_{(t,q)}\in |\mathcal{B}_{\Pi}|_{\leq \sigma}$. The proves the star convexity of $|\mathcal{B}_{\Pi}|_{\leq \sigma}$.\end{proof}
 \subsection*{Proof of Theorem \ref{cellres}}
 \begin{proof}
     From Lemma \ref{starconvexity lemma}, $|\mathcal{B}_{\Pi}|_{\leq \sigma}$ is star-convex for every $\sigma \in \mathbb{N}^{n-1}$. This implies $|\mathcal{B}_{\Pi}|_{\leq \sigma}$ is contractible for every $\sigma$ and hence, acyclicity of $\mathcal{H}_{\Pi}$ follows from \cite[Proposition 4.5]{miller2005combinatorial}. As we know each $C\in \mathcal{B}_{\Pi}$ is an element of $\mathcal{B}_G$ with the same monomial labelling, the minimality of the resolution follows from \cite[Proposition 4]{ds}.
 \end{proof}

We use the preceding theorem to compute algebraic invariants and establish properties for the $G$-parking function ideal $M_{\Pi}$ and the toppling ideal $I_{\Pi}$ of a pargraph $(G,\Pi)$.
 %establish the Cohen-Macaulayness of the toppling and $G$-parking function ideals of a pargraph.

 \begin{corollary} Let $G$ be a connected graph on $n$-vertices and 
let $\Pi$ be a connected $k$-partition of $V(G)$.
     The toppling ideal $I_{\Pi}$ and the $G$-parking function ideal $M_{\Pi}$ of the pargraph $(G,\Pi)$ are Cohen-Macaulay.
 \end{corollary}
\begin{proof} 
    To prove the result, it suffices to prove the Cohen-Macaulayness of $M_{\Pi}$ (Theorem \ref{gboftop}, Proposition \ref{prop3.6}). The toppling ideal $I_{\Pi}$ is a lattice ideal associated with the Laplacian lattice of $L_{\Pi}$ of $(G,\Pi)$, and ${\rm rank}(L_{\Pi})=k-1$. From \cite[Proposition 7.5]{miller2005combinatorial},  the Krull dimension ${\rm dim}(R_n/I_{\Pi})=n-k+1$ and hence, ${\rm dim}(R_n/M_{\Pi})=n-k+1$. 
    As the restricted hyperplane arrangement $\tilde{\mathcal{A}}_{G}$ of the underlying graph is essential and $\tilde{\mathcal{A}}_{\Pi}$ is the restriction of ${\mathcal{A}}_G$ on $U_{\Pi}=U_G \bigcap_{{\{i,j\}\in E_{B}}}h_{ij}$, the restricted arrangement  $\tilde{\mathcal{A}}_{\Pi}$ is essential. The affine space $U_{\Pi}$ has dimension $k-2$, and hence the projective dimension of $R_n /M_{\Pi}$ is $k-1$ (Theorem \ref{cellres}).
    %As $\tilde{\mathcal{A}}_{\Pi}$ is an essential hyperplane arrangement on $U_{\Pi}$ which is an affine space of dimension $k-2$, the projective dimension of $R_n /M_{\Pi}$ is $k-1$ (Theorem \ref{cellres}). 
    From Auslander-Buchsbaum formula, ${\rm depth}(R_n/M_{\Pi})=n-{\rm projdim}(R_n/M_{\Pi})=n-(k-1)$. This proves the result.
\end{proof}
\section{Fano Polytopes and Chip-Firing}\label{FanChip}
This section provides examples of toric ideals arising from polytopes that can be realised as toppling ideals of pargraphs.
A \emph{convex polytope} in $\mathbb{R}^{n}$ is the convex hull of finitely many points in $\mathbb{R}^{n}$. A \emph{Fano polytope} is a full dimensional lattice polytope $P \subset \mathbb{R}^{n}$ such that the origin $\mathbf{0}$ lies in the strict interior of $P$, and every vertex of $P$ is a \emph{primitive lattice point}, i.e. for each vertex $v \in P$, the line segment connecting $\mathbf{0}$ and $v$ contains no lattice point strictly between $\mathbf{0}$ and $v$ \cite{kasprzyk2012fano}.
%A convex lattice polytope is called a \emph{Fano polytope} if the origin lies in the strict interior of the polytope and all of its vertices are \emph{primitive} lattice points, i.e. there is no lattice point between the vertex and the origin. 
\begin{definition} Let $P$ be a convex polytope, and $P \cap \mathbb{Z}^{n}=\{{\bf a_1},\dots,{\bf a_n}\}$ be the set of lattice point of $P$. 
   \begin{enumerate}
       \item The toric ideal $I_P$ associated to $P$ is the kernel of the map $\phi: \mathbb{K}[x_1,\dots,x_n]\newline \rightarrow \mathbb{K}[{\bf t^{a_1}}y,\dots,{\bf t^{a_n}}y]$ given by $\phi(x_i)={\bf t^{a_i}}y$ for all $i$. For a lattice point ${\bf{a_i}} = (a_{i1}, \dots, a_{in}) \in \mathbb{Z}^n$, the corresponding monomial is defined as $\mathbf{t}^{\bf{a}_i} = t_1^{a_{i1}} t_2^{a_{i2}} \cdots t_n^{a_{in}}$.
       \item If $\{\bf a_1,\dots,a_m\}$ is the set of vertices of $P$, the kernel of the map $\psi: \mathbb{K}[x_1,\dots,x_m]\rightarrow \mathbb{K}[{\bf t^{a_1}}y,\dots,{\bf t^{a_m}}y]$ is called the toric ideal of the vertex configurations of $P$.
   \end{enumerate} 
\end{definition}

 \begin{subsection}{Windmill Graph}\label{windG}
 Let $K_m$ be the complete graph on $m$-vertices.
 Consider the \emph{windmill graph} $Wd(m,n)$ obtained by joining $n$ copies of $K_m$ at a shared universal vertex. Let $s$ be the common vertex and $K_{m,i}$ be the $i$-th copy of $K_m$. We denote the vertex set of $K_{m,i}$ as $V(K_{m,i})=\{v_{i1},\dots,v_{i (m-1)}\}\cup \{s\} $. The vertex and edge sets of the windmill graph  $Wd(m,n)$ are $V(Wd(m,n))=\bigcup_{i=1}^{n}V(K_{m,i})$ and $E(Wd(m,n))=\bigcup_{i=1}^{n}E(K_{m,i})$, respectively. \noindent The following are some examples of windmill graphs:

\vspace{0.5cm}

\begin{centering}
% --- FIRST ROW: Wd(3,m) ---
\begin{minipage}{0.32\textwidth}
    \centering
    $Wd(3,3)$\\
    \begin{tikzpicture}[scale=0.8, every node/.style={circle, draw, fill=blue, inner sep=1.5pt}]
        \node (C) at (0,0) {};
        \foreach \i in {0,1,2}{
            \node (n1) at ({120*\i+75}:1.2) {};
            \node (n2) at ({120*\i+105}:1.2) {};
            \draw (C) -- (n1) -- (n2) -- (C);
        }
    \end{tikzpicture}
\end{minipage}
\hfill
\begin{minipage}{0.32\textwidth}
    \centering
    $Wd(4,3)$\\
    \begin{tikzpicture}[scale=0.8, every node/.style={circle, draw, fill=blue, inner sep=1.5pt}]
        \node (C) at (0,0) {};
        \foreach \i in {0,1,2}{
            \begin{scope}[rotate={120*\i+90}]
                \node (n1) at (0.6,1) {};
                \node (n2) at (0,1.5) {};
                \node (n3) at (-0.6,1) {};
                \draw (C)--(n1)--(n2)--(n3)--(C) (n1)--(n3) (C)--(n2);
            \end{scope}
        }
    \end{tikzpicture}
\end{minipage}
\hfill
\begin{minipage}{0.32\textwidth}
    \centering
    $Wd(4,4)$\\
    \begin{tikzpicture}[scale=0.8, every node/.style={circle, draw, fill=blue, inner sep=1.5pt}]
        \node (C) at (0,0) {};
        \foreach \i in {0,1,2,3}{
            \begin{scope}[rotate={90*\i}]
                \node (n1) at (0.6,1) {};
                \node (n2) at (0,1.5) {};
                \node (n3) at (-0.6,1) {};
                \draw (C)--(n1)--(n2)--(n3)--(C) (n1)--(n3) (C)--(n2);
            \end{scope}
        }
    \end{tikzpicture}
\end{minipage}
\end{centering}
\vspace{0.5cm}

Let $({\rm Wd}(m,n),\Pi)$ be the pargraph with $\Pi=\{V(K_{m,i})\setminus \{s\},\{s\}  \mid 1\leq i \leq n\}$. Let $x_{ij}$ be the variable corresponding to the basic vertex $v_{ij}$ of $({\rm Wd}(m,n),\Pi)$ for all $i,j$, and let $x_s$ correspond to the (sink) vertex $s$.
From Theorem \ref{gboftop}, the toppling ideal of $({\rm Wd}(m,n),\Pi)$ is generated by $\{x_{i1}\cdots x_{i(m-1)}-x_s^{m-1}\mid 1\leq i \leq n
\}$ and is a toric ideal (Corollary \ref{corrtoptoricideal}).
\end{subsection}
\subsection{Cross Polytopes and Toppling Ideals}
The $n$-dimensional cross polytope $P_n$ is defined as the convex hull of $\{e_i,- e_i \mid 1\leq i \leq n\}$ in $\mathbb{R}^{n}$ where $e_1,\dots, e_n$ are standard basis vector of $\mathbb{R}^{n}$. Equivalently, $P_n=\{(x_1,\dots,x_n)\in \mathbb{R}^{n} \mid |x_1|+|x_2|+\dots+|x_n|\leq 1\}$. The cross polytope constitutes a fundamental example of a Fano polytope.
\begin{figure}[h]
    \centering
    % --- 2D CROSS-POLYTOPE (Square) ---
    \begin{minipage}{0.45\textwidth}
        \centering
        \begin{tikzpicture}[scale=1.8]
            % Coordinates
            \coordinate (A) at (1,0);
            \coordinate (B) at (0,1);
            \coordinate (C) at (-1,0);
            \coordinate (D) at (0,-1);
            
            % Draw axes
            \draw[->, gray!50] (-1.3,0) -- (1.3,0);
            \draw[->, gray!50] (0,-1.3) -- (0,1.3);
            
            % Draw Shape
            \draw[thick, fill=blue!20, fill opacity=0.6] (A) -- (B) -- (C) -- (D) -- cycle;
            
            % Vertices
            \foreach \p in {A,B,C,D} \fill (\p) circle (1.5pt);
        \end{tikzpicture}
        \caption{2D Cross polytope}
    \end{minipage}
    \hfill
    % --- 3D CROSS-POLYTOPE (Octahedron) ---
    \begin{minipage}{0.45\textwidth}
        \centering
        \tdplotsetmaincoords{75}{115}
        \begin{tikzpicture}[tdplot_main_coords, scale=1.8]
            % Vertices
            \coordinate (P1) at (1,0,0);  \coordinate (P2) at (-1,0,0);
            \coordinate (P3) at (0,1,0);  \coordinate (P4) at (0,-1,0);
            \coordinate (P5) at (0,0,1);  \coordinate (P6) at (0,0,-1);

            % Hidden Edges
            \draw[dashed, gray!60] (P2) -- (P3);
            \draw[dashed, gray!60] (P2) -- (P4);
            \draw[dashed, gray!60] (P2) -- (P5);
            \draw[dashed, gray!60] (P2) -- (P6);

            % Visible Faces
            \fill[blue!20, opacity=0.4] (P1) -- (P3) -- (P5) -- cycle;
            \fill[blue!40, opacity=0.4] (P1) -- (P4) -- (P5) -- cycle;
            \fill[blue!10, opacity=0.4] (P1) -- (P3) -- (P6) -- cycle;
            \fill[blue!50, opacity=0.4] (P1) -- (P4) -- (P6) -- cycle;

            % Visible Edges
            \draw[thick] (P1)--(P3) (P1)--(P4) (P1)--(P5) (P1)--(P6) 
                         (P3)--(P5) (P4)--(P5) (P3)--(P6) (P4)--(P6);
            
            % Points
            \foreach \p in {P1,P3,P4,P5,P6} \fill (\p) circle (1.2pt);
        \end{tikzpicture}
        \caption{3D Cross polytope}
    \end{minipage}
\end{figure}

 For $n \geq 2$, the set of lattice points of the cross polytope $P_n$ is given by
\[
P_n \cap \mathbb{Z}^n = \bigl\{ \mathbf{e}_i,\; -\mathbf{e}_i, {\bf 0}_n  \mid 1 \leq i \leq n \bigr\},
\] where ${\bf 0}_n$ denotes the zero vector in $\mathbb{R}^{n}$.
%For a vector ${\bf u}=(u_1,\dots, u_n)\in \mathbb{Z}^n$, we define ${\bf t^u}=t_1^{u_1}\cdots t_n^{u_n}$.
 %The set of lattice point of $P_n$, for $n\geq 2$, is $\{e_i,- e_i, (0,0,\dots,0) \mid 1\leq i \leq n\}$. 
 Let $\phi$ be the $\mathbb{K}$-algebra homomorphism defined by
$
\phi(x_{i1}) = \mathbf{t}^{\mathbf{e}_i}y, \quad \phi(x_{i2}) = \mathbf{t}^{-\mathbf{e}_i}y \quad \text{for } 1 \leq i \leq n,$
and
$
\phi(x_s) = \mathbf{t}^{\mathbf{0}_n}y = y
$.
%where $\mathbf{e}_1,\dots,\mathbf{e}_n$ denote the standard basis vectors in $\mathbb{R}^n$. 
The toric ideal associated to the cross polytope $P_n$ is given by
$
I_{P_n} = \langle x_{i1}x_{i2} - x_s^2 \mid 1 \leq i \leq n \rangle
$. From Subsection \ref{windG}, $I_{P_n}$ is the toppling ideal of $({\rm Wd}(3,n),\Pi)$.

\subsection{Free Sum of Fano Simplices and Toppling Ideals}

An $n$-dimensional Fano simplex in $\mathbb{R}^{n}$ is an $n$-dimensional lattice simplex with primitive lattice vertices containing the origin in its interior. Let $P\subset \mathbb{R}^{n}$ and $Q \subset \mathbb{R}^{m}$ be two Fano polytopes. The \emph{direct} (or free) sum of $P$ and $Q$ is defined as $P \oplus Q= \text{conv}\{(P \times \{{\bf 0}_m\}) \bigcup (\{{\bf 0}_n\} \times Q)\}\subset \mathbb{R}^{n+m}$. 

Consider the $m$-dimensional Fano simplex $F_m={\rm conv}\{e_1,\dots,e_m,-(e_1+\dots+e_m)\}$ where ${e}_1,\dots,{e}_m$ denote the standard basis vectors in $\mathbb{R}^m$. Let $\bigoplus_{i=1}^{n} F_m$ be the direct sum of $n$ copies of $F_m$. The lattice points of $\bigoplus_{i=1}^{n} F_m$ are \[
\mathcal{L}\left(\bigoplus_{i=1}^n F_m\right) = \{\mathbf{0}_{nm}\} \cup \bigcup_{k=1}^n \left\{ (\mathbf{0}_{m(k-1)}, v, \mathbf{0}_{m(n-k)}) \;\middle|\; v \in \left\{e_1, \dots, e_m, -\sum_{j=1}^m e_j\right\} \right\}.
\]
If we map $x_{ij},x_{i(m+1)}$ and $x_s$ to $t^{(\mathbf{0}_{m(i-1)}, e_j, \mathbf{0}_{m(n-i)})}y, t^{(\mathbf{0}_{m(i-1)}, -\sum_{j=1}^m e_j, \mathbf{0}_{m(n-i)})}y, t^{\mathbf{0}_{nm}}y$, respectively, for all $1\leq j \leq m$ and $1\leq i \leq n$, then
the toric ideal associated with the Fano polytope $\bigoplus_{i=1}^{n} F_m$ is generated by the set $\{x_{i1}\cdots x_{i(m+1)}-x_{s}^{m+1}\mid 1\leq i \leq n
\}$. This ideal is equal to the toppling ideal of the windmill graph $W(m+1,n)$.

\subsection{Grid Graphs and Toric Ideals of Vertex Configurations}
The set of vertices of the free sum of Fano simplices $\bigoplus_{i=1}^{n} F_m$ is $
\mathcal{L}\left(\bigoplus_{i=1}^n F_m\right) \setminus \{\mathbf{0}_{nm}\}$.
If we map $x_{ij},x_{i(m+1)}$ to $t^{(\mathbf{0}_{m(i-1)}, e_j, \mathbf{0}_{m(n-i)})}y, t^{(\mathbf{0}_{m(i-1)}, -\sum_{j=1}^m e_j, \mathbf{0}_{m(n-i)})}y$, respectively, for all $1\leq j \leq m$ and $1\leq i \leq n$, then
the toric ideal is generated by the set $\{x_{i1}\cdots x_{i(m+1)}-x_{(i+1)1}\cdots x_{(i+1)(m+1)}\mid 1\leq i \leq n-1
\}$. For $m=1$ and $n\geq 2$, the toric ideal corresponding to the vertices of  $\bigoplus_{i=1}^{n} F_1$ is the toric ideal of vertex configurations of the cross polytope $P_n$.

Consider the grid graph $G_{n,m}$ with vertex set $\{v_{ij}\mid 1\leq i\leq n, 1\leq j \leq m\}$ (illustrated in Figure \ref{gridgraph}). Let $(G_{n,m},\Pi)$ be the pargraph with partition $\Pi=\{ \{v_{i1},\dots,v_{im}\}\mid 1\leq i \leq n 
\}$. From Theorem \ref{gboftop}, the toppling ideal of the pargraph $(G_{n,m},\Pi)$ is generated by the set $\{x_{i1}\cdots x_{im}-x_{(i+1)1}\cdots x_{(i+1)m}\mid 1\leq i \leq n-1
\}$. This shows that the toric ideal of the vertex configurations of  $\bigoplus_{i=1}^{n} F_{m}$ can be realised as the toppling ideal of the pargraph $(G_{n,m+1},\Pi)$.

\begin{figure}[htbp]
    \centering
    \begin{tikzpicture}[
        % Reduced vertex size to 18pt
        vertex/.style={circle, draw, fill=blue!5, inner sep=0pt, minimum size=18pt, font=\small},
        edge/.style={thick},
        % Cleaner dotted lines without extra math symbols
        dotted edge/.style={thick, loosely dotted, shorten >= 2mm, shorten <= 2mm} 
      ]

      % --- Row 1 (Top) ---
      \node[vertex] (v11) at (0, 0)    {$v_{11}$};
      \node[vertex] (v12) at (1.5, 0)  {$v_{12}$};
      \node[vertex] (v1m) at (3.5, 0)  {$v_{1m}$};

      % --- Row 2 ---
      \node[vertex] (v21) at (0, -1.5)   {$v_{21}$};
      \node[vertex] (v22) at (1.5, -1.5) {$v_{22}$};
      \node[vertex] (v2m) at (3.5, -1.5) {$v_{2m}$};

      % --- Row n (Bottom) ---
      \node[vertex] (vn1) at (0, -3.5)   {$v_{n1}$};
      \node[vertex] (vn2) at (1.5, -3.5) {$v_{n2}$};
      \node[vertex] (vnm) at (3.5, -3.5) {$v_{nm}$};

      % --- Horizontal Edges ---
      \draw[edge] (v11) -- (v12);
      \draw[dotted edge] (v12) -- (v1m);

      \draw[edge] (v21) -- (v22);
      \draw[dotted edge] (v22) -- (v2m);

      \draw[edge] (vn1) -- (vn2);
      \draw[dotted edge] (vn2) -- (vnm);

      % --- Vertical Edges ---
      \draw[edge] (v11) -- (v21);
      \draw[dotted edge] (v21) -- (vn1);

      \draw[edge] (v12) -- (v22);
      \draw[dotted edge] (v22) -- (vn2);

      \draw[edge] (v1m) -- (v2m);
      \draw[dotted edge] (v2m) -- (vnm);

    \end{tikzpicture}
    \caption{The grid graph $G_{n,m}$.}
    \label{gridgraph}
\end{figure}

\footnotesize
\noindent {\bf Author’s address:}
\smallskip
\\
School of Advanced Engineering,\\
UPES Dehradun,\\
India 248007.

\noindent {\bf Email id:}
rahuls.karki@ddn.upes.ac.in, rahulkarki877@gmail.com.

\end{document}